\documentclass[10pt]{article}   
\usepackage{amsthm, amsmath}
\usepackage{amssymb,titlefoot}   
\theoremstyle{plain}

\newtheorem{theorem}{Theorem}[section]
\newtheorem{lemma}[theorem]{Lemma}
\newtheorem{proposition}[theorem]{Proposition}
\newtheorem{corollary}[theorem]{Corollary}

\theoremstyle{definition}

\newtheorem{definition}[theorem]{Definition}

\newtheorem{algorithm}[theorem]{Algorithm}
\newtheorem{notation}[theorem]{Notation}
\newtheorem{remark}[theorem]{Remark}
\newtheorem{parag}[theorem]{}

\DeclareMathOperator{\Sing}{Sing}

\DeclareMathOperator{\EMaxB}{E-\underline{\mathbf{Max}}}
\DeclareMathOperator{\ord}{ord}
\DeclareMathOperator{\Eord}{E-ord}
\DeclareMathOperator{\ESing}{E-Sing}
\DeclareMathOperator{\Etop}{E-top}
\DeclareMathOperator{\ECoeff}{E-Coeff}
\DeclareMathOperator{\BECoeff}{\bf{E-Coeff}}

\DeclareMathOperator{\Img}{Im}

\title{Desingularization of binomial varieties in arbitrary characteristic. \\ Part I. A new resolution function and their properties.}
\author{Roc\'{\i}o Blanco\footnote{Partially supported by MMT2007-64704.}}
\date{}

\begin{document}
\maketitle
\pagestyle{headings}

\begin{abstract}

This paper is devoted to give all the technical  constructions and definitions that will lead to the construction of an algorithm of resolution of singularities for binomial ideals.

We construct a resolution function that will provide a resolution of singularities for binomial ideals, over a field of arbitrary characteristic. For us, a binomial ideal means an ideal generated by binomial equations without any restriction, including monomials and $p$-th powers, where $p$ is the characteristic of the base field. 

This resolution function is based in a modified order function, called $E$-order. The $E$-order of a binomial ideal is the order of the ideal along a normal crossing divisor $E$.

The resolution function allows us to construct an algorithm of $E$-\emph{resolution of binomial basic objects}, that will be a subroutine of the main resolution algorithm.
\end{abstract}

 \\


\section*{Introduction}

In this paper we consider binomial ideals without any kind of restriction, and we construct a resolution function $t$ based in a modified order function, called $E$-order.

The order function and its weighted variants are classically used as useful ingredients to construct resolution functions in the general case, over a field of characteristic zero. See for example Encinas-Villamayor \cite{course} or Encinas-Hauser \cite{strong}.
\medskip

Bierstone and Milman use in \cite{lemabm} the Hilbert Samuel function as resolution function. They construct an algorithm of resolution of singularities, free of characteristic, for reduced binomial ideals with no nilpotent elements. In particular, their algorithm applies to toric ideals.
\medskip

The new resolution function $t$ given in this paper provides combinatorial centers of blowing-up. This type of centers preserve the binomial structure of the ideal after blowing-up, what let us ensure the existence of a hypersurface of maximal contact which to make induction on the dimension of the ambient space at every step of the resolution process.
\medskip

This paper is followed by {\it Desingularization of binomial varieties in arbitrary characteristic. \\ Part II. Combinatorial desingularization algorithm} where we construct an algorithm of resolution of singularities based in this resolution function. The construction of the algorithm, that is the aim of this two articles, needs some technical tools of resolution of singularities. For an easier reading, we present these results in two parts.  This first part is the most technical one and it is not essential to obtain an idea of the running of the algorithm, which is given in Part II. 
\medskip

In section \ref{notdef} we define the $E$-order function, and we prove its main properties in section \ref{propE}.  The resolution function $t$ is defined in section \ref{resfun}. A subroutine of the main algorithm constructed in Part II is given in section \ref{algresbin}.
\medskip

I thank Santiago Encinas for numerous useful suggestions to improve the presentation of this paper. I am also grateful to Antonio Campillo for his help during all this time.  
  
\section{Notation and first definitions} \label{notdef}

Thereafter, $K$ denotes an algebraically closed field (in particular, we will use $K$ is perfect) of arbitrary characteristic, $\textbf{W}$ will be the regular ambient space. 

\begin{remark}
 At any stage of the resolution process, $\textbf{W}=\cup_i U_i$, where $U_i\cong\mathbb{A}^n_K$. In what follows we work locally, inside any affine chart $U_i$, we consider an open set $W$.
\end{remark}

Let $E=\{V_1,\ldots,V_{r}\}$ be a simple normal crossing divisor in $W$. $E$ defines a stratification of $W$ in the following way: we consider the regular closed sets $$E_{\Lambda}=\bigcap_{\lambda\in \Lambda}V_{\lambda} \text{ where } \Lambda\subseteq \{1,\ldots,r\},$$ 
by definition $E_{\emptyset}:=W$, then each $E_{\Lambda}^{0}=E_{\Lambda}\setminus\left((\cup_{j\not\in \Lambda} V_j)\cap E_{\Lambda}\right)$ is locally closed, regular and $$W=\bigsqcup_{\Lambda}E_{\Lambda}^{0}.$$ Therefore, for every $\xi\in W$ there exists a unique $\Lambda(\xi)\subseteq \{1,\ldots,r\}$ such that $\xi\in E_{\Lambda(\xi)}^{0}$.

\begin{remark} At the beginning of the resolution process $W=Spec(K[x_1,\ldots,x_n])$, $dim(W)=n$. Fix the normal crossing divisor $E=\{V_1,\ldots,V_{n}\}$, where $V_i=V(x_i)$ for each $1\leq i \leq n$, to define a stratification of $W$.
\end{remark}

\begin{parag}
Let $J\subset K[x]=K[x_1,\ldots,x_n]$ be a binomial ideal (generated by monomial and binomial equations). Fixed a monomial ordering in $K[x]$, compute a Gr${\rm \ddot{o}}$bner basis of $J$. It is known that the reduced Gr${\rm \ddot{o}}$bner basis of a binomial ideal is binomial, see \cite{sturm}. The analogous result for standard basis is tested for reduced binomial ideals with no nilpotents in \cite{lemabm}. 

Therefore, we start with a binomial ideal $J=<f_1(x),\ldots,f_m(x)>\subset K[x]$ such that the set of generators $\{f_1(x),\ldots,f_m(x)\}$ is the reduced Gr${\rm \ddot{o}}$bner basis of $J$.

After a blowing up $W'\rightarrow W$, binomial equations of the type $$1-\mu x^{\delta}, \text{ with } \mu\in K, \delta\in\mathbb{N}^n$$ appear naturally in the transform ideal of $J$. 
The points $\xi'\in W'$ outside the exceptional divisor where $1-\mu x^{\delta}$ vanishes, satisfy $x^{\delta}(\xi')\neq 0$.  We denote as $y_{i}$ each variable $x_{i}$ that do not vanish anywhere over $V(J)\cap V(1-\mu x^{\delta})$.  And, in what follows, we work in  localized rings of the type $K[x,y]_y$.
\end{parag}

\begin{remark} \label{cartas} At any stage of the resolution process, $\textbf{W}=\cup_i U_i$, where $U_i\cong\mathbb{A}^n_K$, is the regular ambient space over $K$. Inside any chart $U_i$ we consider the open set $$W=Spec(K[x,y]_y)=Spec(K[x_1,\ldots,x_s,y_1,\ldots,y_{n-s}]_y)\subset \mathbb{A}^n_K.$$ 
\end{remark}

\begin{definition} \label{jota}
Let $J\subset K[x,y]_y$ be an ideal. We will say $J$ is a \emph{binomial ideal} if it is generated by binomial equations of the type
\begin{equation} J\!=<\!f_1(x,y),\ldots,f_m(x,y)\!> \text{ with } f_i(x,y)\!=\!x^{\lambda_i}(1-\mu_i y^{\delta_i}) \text{ or } f_j(x,y)\!=\!x^{\nu_j}(y^{\gamma_j}x^{\alpha_j}-b_jx^{\beta_j}),  \end{equation} 
with $\alpha_j,\beta_j\in {\mathbb{N}}^n$, $\delta_i,\gamma_j\in {\mathbb{Z}}^n$, $\lambda_i, \nu_j\in {\mathbb{N}}^n$ and $\mu_i,b_j\in K$ for every $1\leq i,j \leq m$.
And where, for each $j$, every equation of the type $y^{\gamma_j}x^{\alpha_j}-b_j x^{\beta_j}$ has no common factors.\\

Denote $|\alpha_j|=\sum_{k=1}^s \alpha_{j,k}$ and $\ |\beta_j|=\sum_{k=1}^s \beta_{j,k}$.
Assume $0<|\alpha_j|\leq|\beta_j|$. 
\end{definition}

\begin{remark}
Note that the variables $y_i$, $1\leq i \leq n-s$, with $\delta_i\neq 0$ are invertible in the local ring $\mathcal{O}_{W,\xi}$, where $$\xi\in V(1-\mu y^{\delta})\subset \bigcap_{\{i|\ \delta_i\neq 0\}}\{y_i\neq 0\}.$$ 
\end{remark}

\begin{definition} We say $\{\alpha, \beta$, $\gamma\}\in {\mathbb{Z}}^n$, are \emph{disjoint} when they satisfy $$\alpha_l\neq 0 \Rightarrow \beta_l=\gamma_l=0,\  \beta_m\neq 0 \Rightarrow \alpha_m=\gamma_m=0,\   \text{ and } \gamma_t\neq 0 \Rightarrow \alpha_t=\beta_t=0$$ for every $1\leq l,m,t \leq n$.
\end{definition}

That is, in the definition \ref{jota}, $\{\alpha_j, \beta_j,\gamma_j\}$ are disjoint. And,  $\{\alpha_j,\beta_j,\delta_i\}$ are also disjoint. 

\begin{definition} Let $J\subset\mathcal{O}_W$ be a binomial ideal as in \ref{jota}. The binomial equations of $J$ of the form $1-\mu y^{\delta}$, with $\mu \in K$, $\delta \in \mathbb{Z}^n$, are said to be \emph{hyperbolic equations} of $J$.  
\end{definition}

\begin{notation} A unique non hyperbolic binomial equation without common factors, will be denoted 
\begin{equation}
f(x,y)=y^{\gamma}x^{\alpha}-bx^{\beta} \text{ where } \alpha,\beta\in\mathbb{N}^n, \gamma\in\mathbb{Z}^n, b\in K \text{ and } 0<|\alpha|\leq |\beta| \label{efegamma} \end{equation} $\alpha=(\alpha_1,\ldots,\alpha_k,0,\ldots,0)$, $\beta=(0,\ldots,0,\beta_{k+1},\ldots,\beta_{k+(s-k)},0,\ldots,0)$ and $\gamma=(0,\ldots,0,\gamma_{1},\ldots,\gamma_{n-s})$ for some $1\leq k\leq s$.  
\end{notation}

We are going to define a modified order function, the $E$-order, as an order along a normal crossing divisor $E$. This definition works for any ideal, although we will apply it only to binomial ideals.  

\begin{definition} \label{Eord} Let $W=Spec(K[x,y]_y)$ be the regular ambient space. Let $J\subset \mathcal{O}_{W}$ be a binomial ideal as in \ref{jota}. Let $E=\{V_1,\ldots,V_n\}$ be a normal crossing divisor in $\mathbb{A}^n_K$. Let $\xi\in W$ be a closed point and let $\Lambda(\xi)$ be a subset of $\{1,\ldots,n\}$ such that $\xi\in E_{\Lambda(\xi)}^{0}$. We call \emph{$E$-order} of $J_{\xi}$ in $\mathcal{O}_{W,\xi}$ to the order of the ideal with respect to the $I(E_{\Lambda(\xi)}^{0})_{\xi}$-adic topology $$\Eord_{\mathcal{O}_{W,\xi}}(J_{\xi})=\max\left\{m \in \mathbb{N} /\ J_{\xi}\subset (I(E_{\Lambda(\xi)}^{0})_{\xi})^m \right\}.$$ 
\end{definition}

\begin{definition}
Let $J\subset \mathcal{O}_{W}$ be a binomial ideal as in \ref{jota}. Let $\xi \in W$ be a point. The \emph{$E$-order} function (associated to $J$) is defined as follows, $$\begin{array}{rl} \Eord_J: & W\rightarrow \mathbb{N} \\ & \xi \rightarrow \Eord_J(\xi)=\Eord_{\xi}(J):=\Eord_{\mathcal{O}_{W,\xi}}(J_{\xi}) \end{array}$$ The $E$-order of $J$ at $\xi$ will be denoted $\Eord_{\xi}(J)$. The $E$-order of any binomial equation $f\in J$ at $\xi$, is defined as the $E$-order of the ideal $<f>$ at the point $\xi$. 
\end{definition}

\begin{remark} Note that $\Eord_J$ computes the order of the ideal $J$ along $E\cap W=\{V(x_1),\ldots,V(x_s)\}$.
\end{remark}

\begin{remark} Observe that the $E$-order of $f$ is constant along the points in $E_{\Lambda(\xi)}^{0}$. Since to compute the $E$-order are only relevant the coordinates of the point $\xi$ that are contained in some hypersurface $V_j\in E$, $j=1,\ldots,n$.
\end{remark} 

\begin{remark} Since $E_{\emptyset}^{0}=W\setminus (\cup_{i=1}^n V_i\cap W)$ then $I(E_{\emptyset}^{0})=I(W\setminus (\cup_{i=1}^n V_i\cap W))=$ $<0>$. Then for every binomial ideal $J\subset \mathcal{O}_{W}$, $J\neq <0>$, and for each point $\xi \in W\setminus (\cup_{i=1}^n V_i\cap W)$ it holds $\Eord_{\xi}(J)=0$.
\end{remark}

\begin{definition} \label{binbo} An \emph{affine binomial basic object along $E$} (BBOE) is a tuple $B=(W,(J,c),H,E)$ where 
\begin{itemize}
	\item  $W=Spec(K[x,y]_y)=Spec(K[x_1,\ldots,x_s,y_1,\ldots,y_{n-s}]_y)\subset \mathbb{A}^n_K$.
	\item $E$ is the set of normal crossing regular hypersurfaces in $\mathbb{A}^n_K$, such that  $$E=\{V(x_1),\ldots,V(x_s),V(y_1),\ldots,V(y_{n-s})\}.$$ In the open set $Spec(K[x,y]_y)$ we have $E\cap Spec(K[x,y]_y)=\{V(x_1),\ldots,V(x_s)\}$.	
	\item $J$ is a binomial ideal as in (\ref{jota}), and $c$ is a positive integer number.
	\item $H\subset E$ is a set of normal crossing regular hypersurfaces in $W$.
\end{itemize}
\end{definition}  

\begin{definition} A \emph{non affine binomial basic object along $E$} is a tuple $\textbf{B}=(\textbf{W},(\mathcal{J},c),H,E)$ which is covered by affine BBOE. Where 
\begin{itemize}
	\item $\textbf{W}$ is the regular ambient space over a field $K$ of arbitrary characteristic.
	\item $E$ is a set of normal crossing regular hypersurfaces in $\textbf{W}$.
	\item $(\mathcal{J},c)$ is a \emph{binomial pair}, that is, $\mathcal{J}\subset \mathcal{O}_{\textbf{W}}$ is a coherent sheaf of binomial ideals with respect to $E$, as defined in \ref{jota}, satisfying $\mathcal{J}_{\xi}\neq 0$ for each $\xi\in \textbf{W}$, and $c$ is a positive integer number.
	\item $H\subset E$ is a set of normal crossing regular hypersurfaces in $\textbf{W}$.
\end{itemize}
\end{definition}

\begin{remark} In most cases, we work locally with affine BBOE $(W,(J,c),H,E)$ with the notation introduced in \ref{binbo}, without specify that this BBOE is an affine BBOE.
\end{remark}

The definition of \emph{$E$-singular locus} along a normal crossing divisor $E$ is analogous to the usual definition of singular locus. 

\begin{definition} Let $J\subset \mathcal{O}_{W}$ be an ideal, $c$ a positive integer. We call \emph{$E$-singular locus} of $J$ with respect to $c$ to the set, $$\ESing(J,c)=\{\xi \in W/\ \Eord_{\xi}(J)\geq c\}.$$
\end{definition}

\begin{remark} \label{idpair}
Hironaka introduces the notion of equivalence of \emph{pairs} and using this notion, the definition of \emph{idealistic exponent} or \emph{idealistic pair} as an equivalence class of such pairs. See Hironaka \cite{Hironaka1977} for more details.  
\end{remark}

\begin{definition} \label{equiv} Let $W$ be a regular scheme. Let $\mathcal{J}_1,\mathcal{J}_2$ be coherent ideal sheafs in $W$ and let $c_1,c_2$ be positive integers. The pair $(\mathcal{J}_1,c_1)$ is said to be \emph{equivalent} to the pair $(\mathcal{J}_2,c_2)$ if for every morphism $h:Spec(A)\rightarrow W$ where $(A,\nu)$ is a valuation ring, 
$$\frac{\nu(J_1A)}{c_1}=\frac{\nu(J_2A)}{c_2}$$ where $J_iA$ is the ideal in $A$ generated by $J_i$ by means of $h$.
\end{definition}

Therefore, the $E$-order is well defined in the class of idealistic exponents, with the previous definition of equivalence of pairs. 

\begin{remark}\label{equipair}
We always consider pairs $(J,c)$ or binomial basic objects $(W,(J,c),H,E)$ along $E$. This is because of the following result for the order function.
\begin{itemize}
	\item[]  {\small If the tuples $B=(W,(J,c),E)$ and $B_0=(W,(J_0,c_0),E)$ define the same basic object, then for every point $\xi\in \Sing(J,c)=\Sing(J_0,c_0)$, $$\frac{\ord_{\xi}(J)}{c}=\frac{\ord_{\xi}(J_0)}{c_0}.$$ See \cite{Hironaka1977} or \cite{course} for details.}
\end{itemize}
As a consequence, for every point $\xi\in \ESing(J,c)$, the quotient $\frac{\Eord_{\xi}(J)}{c}$ can be defined in terms of the binomial basic object along $E$, modulo the equivalence relation between idealistic exponents. 
\end{remark}
 
\begin{definition} Let $(W,(J,c),H,E)$ be a BBOE, where $W=Spec(K[x_1,\ldots,x_s,y_1,\ldots,y_{n-s}]_y)$ and $E=\{V(x_1),\ldots,V(x_s),V(y_1),\ldots,$ $V(y_{n-s})\}=\{V_1,\ldots,V_n\}$. Let $H=\{H_1,\ldots,H_r\}$ be a normal crossing divisor, $H_i=V(x_j)$ with $1\leq j\leq s$ for each $i$. 

We define a \emph{transformation} of the binomial basic object $$(W,(J,c),H,E)\leftarrow (W',(J',c),H',E')$$ by means of the blowing up $W\stackrel{\pi}{\leftarrow} W'$, along a combinatorial center $Z\subset \ESing(J,c)$, where $W'$ is the strict transform of $W$. With
\begin{itemize}
	\item $H'=\{H_1^{\curlyvee},\ldots,H_r^{\curlyvee},Y'\}$ where $H_i^{\curlyvee}$ is the strict transform of $H_i$ and $Y'$ is the exceptional divisor in $W'$.
	\item $E'=\{V_1^{\curlyvee},\ldots,V_n^{\curlyvee},Y'\}$ where $V_i^{\curlyvee}$ is the strict transform of $V_i$ and $Y'$ is the exceptional divisor in $W'$.
	\item $J'=J^{!}=I(Y')^{\theta-c}\cdot J^{\curlyvee}$ is the \emph{controlled} transform of $J$, where $\theta=\max\ \Eord(J)$ and $J^{\curlyvee}$ is the weak transform of $J$.
\end{itemize}
\end{definition}

\begin{remark} A combinatorial center is given by the intersection of coordinate hypersurfaces defined by variables $x_i$. 
\end{remark}

\begin{definition} A sequence of transformations of binomial basic objects along $E$ ($E^{(0)}=E$)  {\small \begin{equation}   (W^{(0)}\!,(J^{(0)},c),H^{(0)}\!,E^{(0)})\!\leftarrow\!(W^{(1)}\!,(J^{(1)},c),H^{(1)}\!,E^{(1)})\!\leftarrow\!\cdots \leftarrow\!(W^{(N)}\!,(J^{(N)},c),H^{(N)}\!,E^{(N)}) \label{Eresolution} \end{equation}} is a \emph{$E$-resolution} of $(W^{(0)},(J^{(0)},c),H^{(0)},E^{(0)})$ if $\ESing(J^{(N)},c)=\emptyset$. 
\end{definition}

\section{Properties of the $E$-order function} \label{propE}

\begin{definition} \label{semsup} A function $g:X \rightarrow (A,\leq)$, where $X$ is a topological noetherian space and $A$ is a totally ordered set, is said to be \emph{upper semi-continuous} if 
\begin{itemize}
	\item[-] $\Img(g)=\{a_1,\ldots,a_s\}$ is a finite subset of $A$. 
	\item[-] The sets $\mathcal{F}_{a_i}=\{\xi \in X/\ g(\xi)\geq a_i\}$ are closed for all $1\leq i \leq s$.
\end{itemize}
\end{definition}

\begin{remark} It is very well known that the order function is upper semi-continuous. See for example \cite{course}. 
\end{remark}

\begin{proposition} \label{Eorden} Properties of the $E$-order function for binomial equations.
\begin{enumerate}
 	\item[1)] Let $f(x,y)=y^{\gamma}x^{\alpha}-bx^{\beta}$ be a binomial equation without common factors, $\alpha,\beta\in\mathbb{N}^n$, $\gamma\in\mathbb{Z}^n$, $b\in K$ and $0<|\alpha|\leq |\beta|$. Let $\xi\in W$ be a point where $y^{\gamma}(\xi)\neq 0$, then $\Eord_{\xi}(f)\leq \ord_{\xi}(f)$.   
  \item[2)] Let $f(y)=1-\mu y^{\delta}$ be a hyperbolic equation, $\delta\in\mathbb{Z}^n$, $\mu\in K$. Let $\xi \in W$ be a point where $y^{\delta}(\xi)\neq 0$, then $\Eord_{\xi}(f)=0\leq \ord_{\xi}(f)$.  
\end{enumerate}
\end{proposition}

\begin{proof} 
It is enough to observe that $$\Eord_{\xi}(f)=\ord_{E_{\Lambda(\xi)}^{0}}(f)\leq ord_{\xi}(f)$$ by the upper semi-continuity of the order function. Where $\ord_{E_{\Lambda(\xi)}^{0}}(f)$ means the order of $f$ at the generic point of the stratum $E_{\Lambda(\xi)}^{0}$.

Note that in case (1), if $\xi_i=0$ for any $1\leq i\leq n$ with $\alpha_i>0$ or $\beta_i>0$, then the equality holds $\Eord_{\xi}(f)=\ord_{\xi}(f)$.
\end{proof}

\begin{remark} The above proposition still holds for binomial equations with common factors. Since $$\Eord_{\xi}(x^{\nu}(y^{\gamma}x^{\alpha}-bx^{\beta}))=\Eord_{\xi}(x^{\nu})+\Eord_{\xi}
(y^{\gamma}x^{\alpha}-bx^{\beta})$$ and $$\Eord_{\xi}(x^{\lambda}(1-\mu y^{\delta}))=\Eord_{\xi}(x^{\lambda})+\Eord_{\xi}
(1-\mu y^{\delta})=\Eord_{\xi}(x^{\lambda})$$ with $\nu,\lambda\in \mathbb{N}^n$. Note that $\Eord_{\xi}(x^{\nu})=\ord_{\xi}(x^{\nu})$ for any monomial $x^{\nu}$, with $\nu\in \mathbb{N}^n$, $\forall\ \xi\in W$.
\end{remark}

\begin{corollary} Let $J\subset \mathcal{O}_{W}$ be a binomial ideal as in definition \ref{jota}. Then $$\Eord_{\xi}(J)\leq \ord_{\xi}(J)$$ for every $\xi\in W$. 
\end{corollary}

\begin{proof} Let $<f_1,\ldots, f_r>$ be a system of generators of $J$. By the previous proposition  $\Eord_{\xi}(f_i)\leq \ord_{\xi}(f_i)$ for each $1\leq i \leq r$. Then $\Eord_{\xi}(J)\leq \ord_{\xi}(J)$.  
\end{proof} 

\begin{proposition} \label{usc} Let $J\subset \mathcal{O}_{W}$ be a binomial ideal as in definition \ref{jota}, $\xi\in W$. Then $$\begin{array}{cl} \Eord_J: & W \rightarrow (\mathbb{Z},\leq) \\ & \   \xi \rightarrow \Eord_J(\xi):=\Eord_{\xi}(J)\end{array}$$ is an upper semi-continuous function. 
\end{proposition}

\begin{proof} 
The subset $\Img(\ord_J)=\{n_1<\ldots<n_{l}\}$ is a finite subset of $\mathbb{Z}$, since $\ord_J:W \rightarrow (\mathbb{Z},\leq)$ where $\ord_J(\xi)=\ord_{\xi}(J)$, is an upper semi-continuous function. Moreover, $\Eord_{\xi}(J)\leq \ord_{\xi}(J)$ for all $\xi\in W$ because of $J$ is a binomial ideal. Then  $\Img(\Eord_J)=\{m_1<\ldots<m_s\}$ where $m_s\leq n_l$, $m_1\geq 0$.

Let $J=<f_1,\ldots, f_r>$ be a system of generators of $J$. The set $$\mathcal{F}_m(J)=\{\xi \in W/\ \Eord_J(\xi)\geq m\}=\mathcal{F}_m(f_1)\cap \cdots\cap \mathcal{F}_m(f_r).$$ It is enough to check each $\mathcal{F}_m(f_i)$ is a closed set for all $1\leq i \leq r$:

Assume $f_i$ has no common factors. Remind that $\Eord_{\xi}(f_i)=\ord_{\xi}(f_i)$, for all $\xi\in W$, when  $f_i(x)=x^{\alpha}$, with $\alpha\in\mathbb{N}^n$.
\begin{enumerate}  
	\item[1.a)] If $f_i(x,y)=y^{\gamma}x^{\alpha}-bx^{\beta}$ is a binomial equation as in equation (\ref{efegamma}), then $\Eord_{\xi}(f_i)=\ord_{\xi}(f_i)$ for every point $\xi\in W$ satisfying $y^{\gamma}(\xi)\neq 0$ and $\xi_j=0$ for any $1\leq j\leq n$ with $\alpha_j>0$ or $\beta_j>0$.
	
In this case, if $m>0$, $$\mathcal{F}_m(f_i)=\{\xi \in W/ \ \Eord_{\xi}(f_i)\geq m\}=\{\xi \in W/\ \ord_{\xi}(f_i)\geq m\}\cap (\cup_{i} V_i)$$ where $E=\{V_1,\ldots,V_n\}$. Therefore $\mathcal{F}_m(f_i)$ is a closed set since $\ord_J$ is an upper semi-continuous function. 

	\item[1.b)] If $f_i(x,y)=y^{\gamma}x^{\alpha}-bx^{\beta}$ is a binomial equation as in equation (\ref{efegamma}), then $\Eord_{\xi}(f_i)=0$ for every point $\xi\in W$ satisfying $y^{\gamma}(\xi)\neq 0$ and $\xi_j\neq 0$ for all $1\leq j\leq n$ with $\alpha_j>0$ or $\beta_j>0$.
	
	\item[2)] If $f_i(y)=1-\mu y^{\delta}$ is a hyperbolic equation, $\delta\in\mathbb{Z}^n$, $\mu\in K$, then $\Eord_{\xi}(f_i)=0$ for every point $\xi\in W$ where $y^{\delta}(\xi)\neq 0$. 
\end{enumerate}

Observe that the points $\xi\in W$ where $\Eord_{\xi}(f_i)=0$ are only included inside the set $$\mathcal{F}_0(f_i)=\{\xi \in W/\ \Eord_\xi(f_i)\geq 0\}$$ which is a closed set since $\mathcal{F}_0(f_i)=W$. 
	
Therefore, $\mathcal{F}_{m_j}(J)$ is a closed set for all $1\leq j \leq s$. 
\end{proof}

\begin{corollary} Let $J\subset \mathcal{O}_{W}$ be a binomial ideal as in definition \ref{jota}. Let $c$ be a positive integer number. Then $\ESing(J,c)$ is a closed set. 
\end{corollary}

\begin{definition}
Let $t$ be an upper semi-continuous function in $W$. The \emph{top locus} of $t$ is the reduced closed sub-scheme of  $W$ where $t$ reaches its maximum value, that is, $top(t)=\{\xi\in W\ |\ t(\xi)=\max\ t\}.$
\end{definition}

\begin{definition} Let $\mathcal{J}$ be a coherent ideal sheaf in $W$.
\begin{itemize}
	\item[-] The set $\Etop(\mathcal{J})=top(\Eord(\mathcal{J}))$ is said to be the \emph{E-top locus} of $\mathcal{J}$.
	\item[-] Let $c$ be a positive integer number, $\Etop(\mathcal{J},c)=\{\xi\in W\ |\ \Eord_{\xi}(\mathcal{J})\geq c\}.$
\end{itemize}
\end{definition}

\noindent In addition, the $E$-order function is invariant by the torus action.   

\begin{remark}
Let $\varphi$ be the homomorphism of tori 
$$\begin{array}{cccc}\varphi: & \mathcal{T}^d & \rightarrow & \mathcal{T}^n \\ & t & \rightarrow & (t^{a_1},\ldots,t^{a_n}) \end{array}$$ where $\{a_1\ldots,a_n\}\in \mathbb{Z}^d$.

We work with binomial ideals non necessarily toric ideals. Consider the torus action $\mathcal{T}^d$ on $\mathbb{A}^n_K$ given by 
\begin{equation} \label{toro1}
\begin{array}{ccc}\mathcal{T}^d\times \mathbb{A}^n_K & \rightarrow & \mathbb{A}^n_K \\ (t,\xi) & \rightarrow & (t^{a_1}\xi_1,\ldots,t^{a_n}\xi_n) \end{array}
\end{equation} 
extending the homomorphism $\varphi$. 

Note that the torus action (\ref{toro1}) on $\mathbb{A}^n_K$ induces a torus action on the open subset $Spec(K[x,y]_y)\subset\mathbb{A}^n_K$ given by 
\begin{equation} \label{toro2}
\begin{array}{ccc}\mathcal{T}^d\times Spec(K[x,y]_y) & \rightarrow & Spec(K[x,y]_y) \\ (t,\xi) & \rightarrow & (t^{a_1}\xi_1,\ldots,t^{a_n}\xi_n) \end{array}
\end{equation}
\end{remark}

\begin{proposition} \label{equitoro} Let $J\subset \mathcal{O}_{W}$ be a binomial ideal as in (\ref{jota}). Let $X$ be the binomial variety corresponding to the ideal $J$, where $dim(X)=d$. Then $$\Eord_{\xi}(J)=\Eord_{\mathcal{T}^d(\xi)}(J)$$ for all $\xi\in W$, where $\mathcal{T}^d$ is the torus action (\ref{toro2}).
\end{proposition}

\begin{proof} By hypothesis $J$ is given by equations of the form $f(x,y)=y^{\gamma}x^{\alpha}-bx^{\beta}$ or 
$g(y)=(1-\mu y^{\delta})$. Fix $\xi\in W$ such that $\Eord_{\xi}(f)=|\alpha|$. Since $\xi\in W$ then $\mathcal{T}^d(\xi)=(t^{a_1}\xi_1,\ldots,t^{a_n}\xi_n)\in W$ therefore $$f(\mathcal{T}^d(x,y))=t^{\sum_i a_i\gamma_i}y^{\gamma}t^{\sum_i a_i\alpha_i}x^{\alpha}-bt^{\sum_i a_i\beta_i}x^{\beta}$$ so $\Eord_{\mathcal{T}^d(\xi)}(f)=|\alpha|$. The $E$-order remains constant whereas $t_i\neq 0$ for all $i=1,\ldots,d$.

Observe that $\Eord_{\mathcal{T}^d(\xi)}(g)=0=\Eord_{\xi}(g)$. \\ 
This prove the assertion for any generator of $J$, and therefore for any equation belonging to $J$. 
\end{proof}

\begin{remark} \label{equivdef} In what follows, an invariant function by the torus action will be called an \emph{equivariant} function.
\end{remark}

The $E$-order is a suitable function to deal with the problem of the existence of hypersurfaces of maximal contact over fields of positive characteristic. In the particular case of binomial ideals, this modified order function ensures the existence of such hypersurfaces. 

\begin{definition} Let $J\subset \mathcal{O}_W$ be a binomial ideal as in definition \ref{jota}. Let $\xi\in W$ be a point such that $\Eord_{\xi}(J)=\max\ \Eord(J)=\theta$, $V$ is said to be a hypersurface of  \emph{maximal contact along $E$} for $J$ at $\xi$ (denoted by hypersurface of \emph{E\text{-}maximal contact}) if 
\begin{itemize}
	\item [-] $V$ is a regular hypersurface, $\xi\in V$,
	\item [-] $\ESing(J,\theta) \subseteq V$ and their transforms under blowing up along a center $Z\subset V$ also 
	          satisfy $\ESing(J',\theta) \subseteq V'$, where $J'$ is the controlled transform of $J$ and $V'$ is the strict transform of $V$. 
\end{itemize}
\end{definition}  

\begin{remark}
As a consequence, if $V$ is a hypersurface of $E$-maximal contact for $J$ and we have a sequence of blow ups   $$\begin{array}{rrrr} (W,J) & \stackrel{\pi_1}\longleftarrow (W^{(1)},J^{(1)}) & \stackrel{\pi_2}\longleftarrow \cdots & \stackrel{\pi_N}\longleftarrow (W^{(N)},J^{(N)}) \\ Z\subset V & Z^{(1)}\subset V^{(1)} & & Z^{(N)}\subset V^{(N)}  \end{array}$$ with centers $Z^{(i)}\subset V^{(i)}$, where $V^{(i)}$ is the strict transform of $V^{(i-1)}$, satisfying $$\hspace*{-0.5cm} \begin{array}{ccc} \Eord_{\xi}(J)\!=\!\max\ \Eord(J),& \hspace*{-0.2cm} \Eord_{\xi^{(1)}}(J^{(1)})\!=\!\max\ \Eord(J^{(1)}),\ldots,& \hspace*{-0.2cm} \Eord_{\xi^{(N)}}(J^{(N)})\!=\!\max\ \Eord(J^{(N)}) \\ \pi_1(\xi^{(1)})=\xi & \hspace*{-0.2cm} \pi_2(\xi^{(2)})=\xi^{(1)} & \hspace*{-0.2cm} \pi_N(\xi^{(N)})=\xi^{(N-1)} \end{array}$$ and $$\max\ \Eord(J)=\max\ \Eord(J^{(1)})=\cdots=\max\ \Eord(J^{(N)})=\theta$$ then $$\ESing(J^{(1)},\theta)\subseteq V^{(1)}, \ldots, \ESing(J^{(N)},\theta)\subseteq V^{(N)}.$$ The strict transforms of $V$ preserve this property while the maximum $E$-order remains constant. 

\end{remark}

\begin{definition} Let $f$ be a binomial equation in $\mathcal{O}_{W}$. Let $a\in W$ be a point. The 
\emph{equimultiple locus of $a$ along $E$} for $f$ is the set of points where $f$ has the same $E$-order as at the point $a$  $$S_{f,E}(a)=\{\xi\in W\ | \ \Eord_{\xi}(f)=\Eord_{a}(f) \}.$$  
\end{definition}

\begin{theorem} \label{gen} Let $f(x,y)=y^{\gamma}x^{\alpha}-bx^{\beta}\in K[x,y]_y$ be a binomial equation as in equation (\ref{efegamma}). Let $a\in Spec(K[x,y]_y)$ be a point satisfying $a_i=0$ for all $i$ with $\alpha_i>0$, $y^{\gamma}(a)\neq 0$ and $\Eord_a(f)=|\alpha|$. 

If $|\alpha|>0$ then $$S_{f,E}(a)\subset \bigcap_{\{i|\ \alpha_i>0\}}\{x_i=0\}.$$
\end{theorem}

\begin{proof} Let $\xi\in S_{f,E}(a)$ be a point verifying $y^{\gamma}(\xi)\neq 0$, $\Eord_{\xi}(f)=|\alpha|$. Suppose $\alpha_i>0$ for every $1\leq i\leq k$. If $\xi\notin \{x_1=\ldots =x_k=0\}$ then $\xi_i\neq 0$ for some $1\leq i \leq k$. 

Let assume $\xi$ verifies $\xi_k\neq 0$, $\xi_1=\ldots=\xi_{k-1}=0$.\\ The $E$-order can be computed in terms of the exponents, $\Eord_{\xi}(f)=min\{|\alpha|_{\xi},|\beta|_{\xi}\}$ where $$|\alpha|_{\xi}\!=\!\sum_{\{i|\ \xi_i=0,\alpha_i>0\}} \alpha_i \text{\ and \ } |\beta|_{\xi}\!=\!\sum_{\{i|\ \xi_i=0,\beta_i>0\}} \beta_i.$$ If $\xi_k\neq 0$ then $\Eord_{\xi}(f)\leq \alpha_1+\ldots+\alpha_{k-1}<|\alpha|$, that is, by definition of $E$-order, $$<f>_{\xi}\subset (I(E^0_{\Lambda(\xi)})_{\xi})^{|\alpha|-\alpha_k}=<x_1,\ldots,x_{k-1},x_{k+i_1},\ldots,x_{k+i_s}>^{|\alpha|-\alpha_k}$$ but $$<f>_{\xi}\not\subset <x_1,\ldots,x_{k-1},x_{k+i_1},\ldots,x_{k+i_s}>^{|\alpha|}$$ for $i_j$ such that  $\beta_{i_j}>0$ and $\xi_{i_j}=0$ for each $1\leq j\leq s$, and $\sum_{j=1}^s\beta_{i_j}\geq |\alpha|$, contradiction.

In the same way, if $\xi_1=\ldots=\xi_{l}=0$ and $\xi_j\neq 0$ for all $l+1\leq j\leq k$ then $\Eord_{\xi}(f)=|\alpha|-\sum_{j=l+1}^{k}\alpha_j$, contradiction.

If $\xi_j\neq 0$ for all $1\leq j \leq k$, then $\Eord_{\xi}(f)=0$ contradiction.
\end{proof}

\begin{corollary} \label{esing} Let $J\subset K[x,y]_y$ be a binomial ideal as in definition \ref{jota}. Let $\xi \in Spec(K[x,y]_y)$ be a point where $\Eord_{\xi}(J)=\theta>0$ is maximal. Let $f\in J$ be a binomial equation such that  $\Eord_{\xi}(f)=\theta$ and $f$ satisfies the hypothesis of the Theorem \ref{gen}. Then in a neighborhood of $\xi$, $$\ESing(J,\theta) \subseteq \{x_i=0\} \text{ for some } 1\leq i\leq k. $$  
\end{corollary}

\begin{proof} It is enough to observe that $\ESing(J,\theta)\subseteq S_{f,E}(\xi)$.
\end{proof}

\begin{remark}
As a consequence, the hypersurfaces of $E$-maximal contact will always be given by coordinate equations. The existence of these hypersurfaces will be proved in lemma \ref{hiper}. Hence the centers of blowing up will always be combinatorial. 
\end{remark}

\section{$E$-resolution function} \label{resfun}

\subsection{Induction on the dimension} \label{subsec:centro}

In this section we rewrite \emph{mobiles language} in order to use this language with the $E$-order function, and to make induction on the dimension of the ambient space. See \cite{strong} for more details.

Given $(W,(J,c),H,E)$ a binomial basic object along $E$, by induction on the dimension of $W$, construct ideals $J_i$ defined in local flags $W=W_n\supseteq W_{n-1}\supseteq \cdots \supseteq W_i\supseteq  \cdots \supseteq W_1$, and then binomial basic objects $(W_i,(J_i,c_{i+1}),H_i,E_i)$ in dimension $i$, where each $E_i=W_i\cap E$. 

\begin{remark} \label{fac}
If $\ESing(J_i,c_{i+1})\neq\emptyset$ then factorize the ideal $J_i=M_i\cdot I_i$, where each ideal $M_i$ is defined by a normal crossings divisor $D_i$ supported by the current exceptional locus. 
\end{remark}
 
In what follows we will define the ideals $J_{i-1}$, $n\geq i> 1$. We need the auxiliary definition of the companion ideals $P_i$. 

\begin{definition} Let $J_i=M_i\cdot I_i$ be an ideal in $W_i$ at $\xi\in W_i$. Set $\theta_i=\Eord_{\xi}(I_i)$. 
The {\it companion ideal} of $J_i$ at $\xi$, with respect to the critical value $c_{i+1}$ satisfying $\Eord_{\xi}(J_i)\geq c_{i+1}$, is the ideal 
$$P_i= \left\{\begin{array}{ll} I_i & \text{ if }\ \theta_i\geq
c_{i+1} \\ I_i+M_i^{\frac{\theta_i}{c_{i+1}-\theta_i}} & \text{ if }\ 0< \theta_i< c_{i+1} \end{array}\right.$$ 
\end{definition}   

\begin{remark}
If $\theta_i=0$ but $\ord_{\xi}(I_i)\neq 0$ then there exists some hyperbolic equation belonging to the ideal $I_i$. In this case, the ideal $I_i\neq 1$ but it will be treated as $I_i\equiv 1$, since $J_i$ behaves like a monomial ideal with respect to the $E$-order. So, in this situation, $J_i\equiv M_i$.    
\end{remark}

\begin{remark}   
On one hand, the weight assigned to the ideal $M_i$ ensures that $\Eord(P_i)=\Eord(I_i)$ along the points of maximal $E$-order, where $\Etop(P_i)\subset \Etop(I_i)$. So $P_i=1$ if and only if $I_i=1$ or $I_i\equiv 1$. On the other hand, $$\Etop(P_i)=\Etop(I_i)\cap \Etop(M_i,c_{i+1}-\theta_i)\subset \Etop(J_i,c_{i+1})$$ what ensures $\ESing(P_i,\theta_i)\subset \ESing(J_i,c_{i+1})$ when $0< \theta_i< c_{i+1}$. 
\end{remark}

\begin{remark} In addition, this weight guarantees the weak transform of $P_i$ by a blow up $\pi$ satisfies  $$P_i^{\curlyvee}=(I_i+M_i^{\frac{\theta_i}{c_{i+1}-\theta_i}})^{\curlyvee}=I_i^{\curlyvee}+(M_i')^{\frac{\theta'_i}{c_{i+1}'-\theta'_i}}$$ when $\theta'_i=\theta_i$ and $c_{i+1}'=c_{i+1}$, where $\theta'_i=\Eord_{\xi'}(I_i')$ in a neighborhood of the point $\xi'$ such that $\pi(\xi')=\xi$.
\end{remark}

\begin{definition} Let $J_i$ be an ideal in $W_i$ at $\xi\in W_i$. Let $P_i$ be the companion ideal of $J_i$ in a neighborhood of the point $\xi$ with respect to the critical value $c_{i+1}$.

The critical value $c_i$ corresponding to dimension $i-1$ is $c_i=\Eord_{\xi}(P_i)$. 
\end{definition} 

\begin{definition} \label{Ecoefideal} Let $P$ be an ideal in $W$, let $V\subset W$ be a regular hypersurface, and let $\xi\in V$ be a point. Let $\{z,w\}$ be a regular system of parameters of $\mathcal{O}_{W,\xi}$ and let $\{w\}$ be a regular system of parameters of $\mathcal{O}_{V,\xi}$ such that $\{z=0\}$ defines $V$ in $W$. 

For all $f\in P$, let $f=\sum_{\alpha}a_{f,\alpha}z^{\alpha}$ be its Taylor expansion with respect to the equation defining $V$, so that $a_{f,\alpha}\in \mathcal{O}_{V,\xi}$ and this equality holds after passage to the completion.

The \emph{coefficient ideal of $P$ along $E$} at $\xi$ with respect to $V$ is the ideal $$\ECoeff_V(P)=<\ECoeff_V(f),f\in P>=\sum_{f\in P, |\alpha|<c}(a_{f,\alpha})^{\frac{c}{c\text{-}|\alpha|}}$$ where $c$ is the suitable critical value. 
\end{definition}

\begin{remark} To avoid rational exponents take $\frac{c!}{c\text{-}|\alpha|}$ instead of $\frac{c}{c\text{-}|\alpha|}$ in the previous definition.  

The pair $({\BECoeff}_V(P),c!)$ whose exponents are integer numbers is equivalent to the pair $(\ECoeff_V(P),c)$ with rational exponents. 
\end{remark}

\begin{proposition} \label{Ecoeford} Let $P$ be an ideal in $W$, let $V\subset W$ be a regular hypersurface. Then \begin{equation} \Eord_{\xi}(\ECoeff_V(P))\geq \Eord_{\xi}(P) \label{desigual}
\end{equation} for all point $\xi\in V$.  
\end{proposition}

\begin{proof} Let $\xi\in V$ be a point. Set $c=\Eord_{\xi}(P)$. This means $\Eord_{\xi}(f)\geq c$ for all $f\in P$, so that $\Eord_{\xi}(a_{f,\alpha})\geq c-|\alpha|$ for all $f\in P$. 

Then $$\Eord_{\xi}\left((a_{f,\alpha})^{\frac{c}{c-|\alpha|}}\right)\geq (c-|\alpha|)\cdot\left(\frac{c}{c-|\alpha|}\right)\ \text{ for all } f\in P,$$ 
therefore $$\Eord_{\xi}(\ECoeff_V(P))=\min_{\forall f\in P} \left(\Eord_{\xi}\left(\sum_{|\alpha|<c}(a_{f,\alpha})^{\frac{c}{c-|\alpha|}} \right)\right)\geq c.$$  
\end{proof}

\begin{remark} Let $P$ be a binomial ideal in $W$. Every hypersurface of $E$-maximal contact for $P$ at any point $a\in W$ is achieved by Theorem \ref{gen}. 

Under these conditions, if $\{x_1=0\}$ is a hypersurface of $E$-maximal contact for $P$ at $a$, then there exists a binomial equation $f(x,y)=y^{\gamma}x^{\alpha}-bx^{\beta} \in P$ with $\alpha_1>0$ such that $\Eord_{a}(f)=|\alpha|=\Eord_a(P)$, and therefore $\{x_1=0\}$ is a hypersurface of $E$-maximal contact for $f$. 
\end{remark}

\begin{lemma} Let $P$ be a binomial ideal in $W$. If $\{x_1=0\},\{x_2=0\}$ are hypersurfaces of $E$-maximal contact for $P$ in a neighborhood of $a\in W$, then,
\begin{enumerate}
  \item[a)] The hypersurface $\{x_2=0\}$ is a hypersurface of $E$-maximal contact for $\ECoeff_{\{x_1=0\}}(P)$. \\ Analogously, $\{x_1=0\}$ is a hypersurface of $E$-maximal contact for $\ECoeff_{\{x_2=0\}}(P)$.
	\item[b)] $\ECoeff_{\{x_2=0\}}\left(\ECoeff_{\{x_1=0\}}(P)\right)=\ECoeff_{\{x_1=0\}}\left(\ECoeff_{\{x_2=0\}}(P)\right)$.
\end{enumerate}
\end{lemma}

\begin{proof} \
\begin{itemize}
	\item[a)] Fix $c=\max\ \Eord(P)$. Let $a\in W$ be a point where $\Eord_{a}(P)=c$.	Let $f\in P$ be a binomial equation, without common factors, such that $\Eord_{a}(f)=c$ and $\{x_1=0\}$ is a hypersurface of $E$-maximal contact for $f$ at $a$. 

\item If $\{x_2=0\}$ is also a hypersurface of $E$-maximal contact for $f$ at $a$ then there are two cases: 
\begin{itemize}
	\item[-] If $f$ is of the form $f(x,y)=y^{\gamma}x_1^{\alpha_1}x_2^{\alpha_2}x^{\alpha*}-bx^{\beta}$, with $\alpha_1>0,\alpha_2>0$, $\alpha*=\alpha-(\alpha_1,\alpha_2,0\ldots,0)$ and $c=|\alpha|\geq\alpha_1+\alpha_2$. The coefficient ideal of $<f>$ along $E$ is,  $$\ECoeff_{\{x_1=0\}}(<f>)=<(y^{\gamma}x_2^{\alpha_2}x^{\alpha*})^{\frac{c}{c\text{-}\alpha_1}},bx^{\beta}>.$$ Since $|\beta|\geq |\alpha|=c$, $\Eord_a(\ECoeff_{\{x_1=0\}}(<f>))=|\alpha|=c$ and therefore $$\Eord_a(\ECoeff_{\{x_1=0\}}(P))=|\alpha|=c.$$ So that $\{x_2=0\}$ is a hypersurface of $E$-maximal contact for $\ECoeff_{\{x_1=0\}}(P)$ at $a$.
\item[-] If $f$ is of the form $f(x,y)=y^{\gamma}x_1^{\alpha_1}x^{\alpha*}-bx_2^{\beta_2}x^{\beta*}$, with $\alpha*=\alpha-(\alpha_1,0\ldots,0)$, $\beta*=\beta-(0,\beta_2,0\ldots,0)$, $\alpha_1>0,\beta_2>0$ and $c=|\alpha|=|\beta|$. The coefficient ideal of $<f>$ along $E$ is,   $$\ECoeff_{\{x_1=0\}}(<f>)=\left\{\begin{array}{ll}<(y^{\gamma}x^{\alpha*})^{\frac{c}{c\text{-}\alpha_1}},bx_2^{\beta_2}x^{\beta*}> & \text{if } \alpha_1<c \\ <bx_2^{\beta_2}x^{\beta*}> & \text{if } \alpha_1=c \end{array}\right.$$ Again $\Eord_a(\ECoeff_{\{x_1=0\}}(<f>))=|\alpha|=c$ so that $\{x_2=0\}$ is a hypersurface of $E$-maximal contact for $\ECoeff_{\{x_1=0\}}(P)$ at $a$.
\end{itemize}

\item If $\{x_2=0\}$ is not a hypersurface of $E$-maximal contact for $f$ at $a$, then there exists a binomial equation $g\in P$ such that $\Eord_{a}(g)=c$ and $\{x_2=0\}$ is a hypersurface of $E$-maximal contact for $g$ at $a$. This equation $g$ is of the form $g(x,y)=y^{\delta}x_2^{\eta_2}x^{\eta*}-bx^{\mu}$, without common factors, where $\eta*=\eta-(0,\eta_2,0\ldots,0)$, $\eta_1+\eta_2\leq|\eta|$ and $c=|\eta|\leq |\mu|$. 
The coefficient ideal of $<g>$ along $E$ is, $$\ECoeff_{\{x_1=0\}}(<g>)=\left\{\begin{array}{ll} <g> & \text{if } \eta_1=\mu_1=0 \\ <(y^{\delta}x_2^{\eta_2}x^{\eta^{**}})^{\frac{c}{c\text{-}\eta_1}},bx^{\mu}> & \text{if } \eta_1>0 \\ <y^{\delta}x_2^{\eta_2}x^{\eta*},(bx^{\mu*})^{\frac{c}{c\text{-}\mu_1}}> & \text{if } 0<\mu_1<c \\ <y^{\delta}x_2^{\eta_2}x^{\eta*}> & \text{if } \mu_1\geq c \end{array}\right.$$ where $\eta^{**}=\eta^{*}-(\eta_1,0\ldots,0)$, $\mu^{*}=\mu-(\mu_1,0,\ldots,0)$. 

So that $\Eord_a(\ECoeff_{\{x_1=0\}}(<g>))=|\eta|=c$ and therefore $\{x_2=0\}$ is a hypersurface of $E$-maximal contact for $\ECoeff_{\{x_1=0\}}(P)$ at $a$. 	

\item To check the assumption when $f$ has common factors is an easy exercise. 
	\item[b)] It is enough to check {\small $$ \ECoeff_{\{x_2=0\}}\!\left(\ECoeff_{\{x_1=0\}}(<\!f\!>)\right)\!=\!\ECoeff_{\{x_1=0\}}\!\left(\ECoeff_{\{x_2=0\}}(<\!f\!>)\right)\  \forall\ f\in P.$$}
In $(a)$ it is proved that the equality in equation (\ref{desigual}) holds if there exist at least two hypersurfaces of $E$-maximal contact for $P$ at $a\in W$. 
	
Let $f(x,y)=y^{\gamma}x^{\alpha}-bx^{\beta}$ be a binomial equation, $f\in P$, with $|\beta|\geq|\alpha|$, where $|\alpha|\geq c$. 

\item Assume $\alpha_1<c$. The coefficient ideal of $<f>$ along $E$ is given by, $$J_f^{1}\!=\!\ECoeff_{\{x_1=0\}}(<f>)\!=\!\left\{\begin{array}{ll} \hspace*{-0.2cm} <f> & {\scriptstyle \text{if } \alpha_1=\beta_1=0} \\ \hspace*{-0.2cm}
<(y^{\gamma}x^{\alpha^{*}}-bx^{\beta^{*}})^{\frac{c}{c\text{-}\alpha_1}}> & {\scriptstyle \text{if } \alpha_1=\beta_1\neq 0} \\ \hspace*{-0.2cm} <(y^{\gamma}x^{\alpha^{*}})^{\frac{c}{c\text{-}\alpha_1}}, (bx^{\beta^{*}})^{\frac{c}{c\text{-}\beta_1}}> & {\scriptstyle \text{if } 0<\beta_1<c,\ \beta_1\neq \alpha_1} \\ \hspace*{-0.2cm} <(y^{\gamma}x^{\alpha^{*}})^{\frac{c}{c\text{-}\alpha_1}}> & 
{\scriptstyle \text{if } \beta_1\geq c} \\ 
\end{array}\right.$$

where $\alpha^{*}=\alpha-(\alpha_1,0,\ldots,0)$, $\beta^{*}=\beta-(\beta_1,0,\ldots,0)$.  

Note that $\max\ \Eord(J_f^{1})\geq c$, but $\max\ \Eord(\ECoeff_{\{x_1=0\}}(P))=c$ because of $(a)$. \\ Notice that $\alpha_1<c$ and $\alpha_2\cdot\frac{c}{c-\alpha_1}<c$ are equivalent to $\alpha_1+\alpha_2<c$.

Since $\ \frac{c}{c-\alpha_1}\cdot \frac{c}{c-(\frac{\alpha_2\cdot
c}{c-\alpha_1})}=\frac{c}{c-(\alpha_1+\alpha_2)}$  and $\
\frac{c}{c-\beta_1}\cdot \frac{c}{c-(\frac{\beta_2\cdot
c}{c-\beta_1})}=\frac{c}{c-(\beta_1+\beta_2)}$,

{\small $$\ECoeff_{\{x_2=0\}}(J_f^{1})=\left\{\begin{array}{ll} \hspace*{-0.2cm} <f> & {\scriptstyle \text{if } \alpha_1=\beta_1=\alpha_2=\beta_2=0} \\ \hspace*{-0.2cm}
<(y^{\gamma}x^{\alpha'}-bx^{\beta'})^{\frac{c}{c-\alpha_2}}> & {\scriptstyle \text{if } \alpha_1=\beta_1=0, \alpha_2=\beta_2\neq 0} \\ \hspace*{-0.2cm} <(y^{\gamma}x^{\alpha^{**}}-bx^{\beta^{**}})^{\frac{c}{c-(\alpha_1+\alpha_2)}}> & {\scriptstyle \text{if } \alpha_1=\beta_1\neq 0,\alpha_2=\beta_2\neq 0} \\ \hspace*{-0.2cm}
<(y^{\gamma}x^{\alpha^{*}}-bx^{\beta^{*}})^{\frac{c}{c-\alpha_1}}> & {\scriptstyle \text{if } \alpha_1=\beta_1\neq 0,\alpha_2=\beta_2=0} \vspace*{0.1cm} \\ \hspace*{-0.2cm}
<(y^{\gamma}x^{\alpha^{**}})^{\frac{c}{c-(\alpha_1+\alpha_2)}}, (bx^{\beta^{**}})^{\frac{c}{c-(\beta_1+\beta_2)}}> & {\scriptstyle \text{if }} \begin{array}{l} {\scriptstyle
0<\beta_1+\beta_2<c,} \\ {\scriptstyle \beta_1\neq \alpha_1 \text{ or } \beta_2\neq \alpha_2} \end{array} \\ \hspace*{-0.2cm} <(y^{\gamma}x^{\alpha^{**}})^{\frac{c}{c-(\alpha_1+\alpha_2)}}> & {\scriptstyle \text{if } \beta_1\geq c \text{ or } \beta_2\geq c} \\ 
\end{array}\right.$$}

where $\alpha^{**}=\alpha-(\alpha_1,\alpha_2,0,\ldots,0)$, $\beta^{**}=\beta-(\beta_1,\beta_2,0,\ldots,0)$,  $\alpha^{'}=\alpha-(0,\alpha_2,0,\ldots,0)$ and $\beta^{'}=\beta-(0,\beta_2,0,\ldots,0)$.

Replacing $x_1$ by $x_2$, $\alpha_1$ by $\alpha_2$ and $\beta_1$ by $\beta_2$ in the expression of $J_f^{1}$ obtain $$J_f^{2}=\ECoeff_{\{x_2=0\}}(<f>).$$ It follows that $\ECoeff_{\{x_1=0\}}(J_f^{2})=\ECoeff_{\{x_2=0\}}(J_f^{1})$.  

\item If $\alpha_1\geq c$,  $$J_f^{1}=\ECoeff_{\{x_1=0\}}(<f>)=\left\{\begin{array}{ll} 
<(bx^{\beta^{*}})^{\frac{c}{c-\beta_1}}> & \text{if } \beta_1<c \\ 0 & \text{if } \beta_1\geq c \end{array}\right.$$ where $\beta^{*}=\beta-(\beta_1,0,\ldots,0)$. The exponent of $x_2$, $\beta_2\cdot\frac{c}{c-\beta_1}<c$ if and only if $\beta_1+\beta_2<c$. 

As above
$$\ECoeff_{\{x_2=0\}}(J_f^{1})=\left\{\begin{array}{ll} 
<(bx^{\beta^{**}})^{\frac{c}{c-(\beta_1+\beta_2)}}> & \text{if } \beta_1<c, \beta_1+\beta_2<c \\ 0 & \text{otherwise} \end{array}\right.$$ with $\beta^{**}=\beta-(\beta_1,\beta_2,0,\ldots,0)$. And $$\ECoeff_{\{x_1=0\}}(J_f^{2})=\left\{\begin{array}{ll} 
<(bx^{\beta^{**}})^{\frac{c}{c-(\beta_1+\beta_2)}}> & \text{if } \beta_2<c, \beta_1+\beta_2<c \\ 0 & \text{otherwise} \end{array}\right.$$ so that $\ECoeff_{\{x_1=0\}}(J_f^{2})=\ECoeff_{\{x_2=0\}}(J_f^{1})$.

\item Following the previous argument, compute $\ECoeff_{\{x_2=0\}}(J_f^{1})$ when $\alpha_1+\alpha_2\geq c$. In this case $\alpha_2\cdot\frac{c}{c-\alpha_1}\geq c$ and the coefficient ideals are the same as in the case $\alpha_1\geq c$.
\end{itemize} \vspace*{-0.5cm}
\end{proof}
	
\begin{proposition} \label{topes} Let $P$ be an ideal in $W$. Let $V=\{z=0\}\subset W$ be a hypersurface of $E$-maximal contact for $P$ at $\xi\in V$. Fix $c=\Eord_{\xi}(P)=\max\ \Eord(P)$. Then, $$\Etop(P)=\Etop(\ECoeff_V(P),c).$$  
\end{proposition}

\begin{proof} By proposition \ref{Ecoeford} $\Etop(P)\subset \Etop(\ECoeff_V(P),c)$.
	
Let $\eta \in V$ such that $\eta\in \Etop(\ECoeff_V(P),c)$. By definition, 	$$\ECoeff_V(P)=<\{(a_{f,\alpha})^{\frac{c}{c-|\alpha|}}\ /\ f\in P, |\alpha|<c\}>$$
where $V=\{z=0\}$ and  $f=\sum_{\alpha}a_{f,\alpha}z^{\alpha}$, $\forall\ f\in P$.
$$ \begin{array}{rl} \Eord_{\eta}(\ECoeff_V(P))\geq c  & \Leftrightarrow \Eord_{\eta}((a_{f,\alpha})^{\frac{c}{c-|\alpha|}})\geq c, \ \forall\ f\in P \text{ and } \ \forall\ |\alpha|<c \Leftrightarrow \\  & \Leftrightarrow \Eord_{\eta}(a_{f,\alpha})\geq c-|\alpha|, \ \forall\ f\in P \text{ and } \ \forall\ |\alpha|<c \Leftrightarrow \\  & \Leftrightarrow \Eord_{\eta}(f)\geq c, \ \forall\ f\in P \end{array}$$ because of $\ESing(f,c)\subset V$. 

Since $c$ is maximum, $\Eord_{\eta}(f)=c$ for all $f\in P$ and then $\eta\in \Etop(P)$.
\end{proof}

\begin{proposition} \label{conmutdebil} Let $P$ be an ideal in $W$. Let $V\subset W$ be a hypersurface of E\text{-}maximal contact for $P$ at $\xi\in V$. The coefficient ideal of $P$ along $E$ satisfies $$(\ECoeff_V(P))^{!}=\ECoeff_{V^{\curlyvee}}(P^{\curlyvee})$$ after a blow up $\pi$ along $Z\subset \Etop(P)$, at those points where $$c'=\Eord_{\xi'}(P^{\curlyvee})=\Eord_{\xi}(P)=c,$$ with $\pi(\xi')=\xi$, $P^{\curlyvee}$ is the weak transform of $P$ and the superscript $^{!}$ means controlled transform.
\end{proposition}

\begin{proof} The proof is word by word the same as in \cite{strong} rewritten for the $E$-order.  
\end{proof}

\begin{definition} Let $P\neq 0$ be an ideal in a hypersurface $V=\{x_1=0\}\subset W=Spec(K[x,y]_y)$. The ideal $P$ is said to be \emph{bold regular along $E$} or \emph{E\text{-}bold regular} if 
 $P$ is of the form $P=<y^{\gamma}(1-\mu y^{\delta})x_1^{\alpha_1}>$ where $\mu \in K$, $\gamma,\delta\in \mathbb{Z}^n$, $\alpha_1\in \mathbb{N}$, $\alpha_1>0$.
\end{definition}

\begin{remark} The factor $y^{\gamma}(1-\mu y^{\delta})$ satisfies $$\Eord_{\xi}(y^{\gamma}(1-\mu y^{\delta}))=0 \text{ for all }\xi \in V\cap \{\eta \in W|\ y^{\gamma}(\eta)\neq 0, y^{\delta}(\eta)\neq 0\}.$$
\end{remark}
 
\begin{proposition} \label{boldreg} Let $P\neq 0$ be an ideal in $W$. Let $V\subset W$ be a hypersurface of $E$-maximal contact for $P$ in a neighborhood of $\xi\in V$, with $\Eord_{\xi}(P)=c$. 

Then $\ECoeff_V(P)=0$ if and only if $P$ is bold regular or $1$.
\end{proposition}

\begin{proof} \ 
\begin{itemize}
	\item If the ideal $P$ is bold regular, $P=<y^{\gamma}(1-\mu y^{\delta})x_1^{\alpha_1}>$. Then $\Etop(P)=V=\{x_1=0\}$ and $\ECoeff_{V}(P)=0$. If $P=1$ then there no exists any hypersurface $V$ of maximal contact for $P$, $\ECoeff_{V}(P)=0$ by definition. 

\item If $V=\{x_1=0\}$, $\ECoeff_{V}(P)=0$ and $P\neq 1$ then the coefficients $a_{f,\alpha}=0$ for all $f\in P$ where $f=\sum_{\{\alpha/|\alpha|<c\}}a_{f,\alpha}x_1^{\alpha}$. So $P\subset <x_1^{c}>$. 

Since $\Eord_{\xi}(P)=c$ at $\xi\in V$, then $P=<\mathcal{U}\cdot x_1^{c}>$ where $\Eord_{\xi}(\mathcal{U})=0$. Therefore $\mathcal{U}$ is of the form $\mathcal{U}=y^{\gamma}(1-\mu y^{\delta})$ with $y^{\gamma}(\xi)\neq 0$, $y^{\delta}(\xi)\neq 0$, $\mu \in K$, $\gamma,\delta\in \mathbb{Z}^n$.
\end{itemize} \vspace*{-0.5cm}
\end{proof}

\begin{definition} Let $P_i$ be an ideal in $W_i$. Let $V\subset W_i$ be a hypersurface of $E$-maximal contact for $P_i$ at $\xi \in V$. Denote $\Eord_{\xi}(P_i)=c_i$. The {\it junior ideal} of $P_i$ in $V$ is the ideal
$$J_{i-1}= \left\{
\begin{array}{ll} \ECoeff_V(P_i) & \text{ if } \ECoeff_V(P_i)\neq 0 \\ 1 & \text{ if } \ECoeff_V(P_i)=0 \end{array} \right. $$ where $\ECoeff_{V}(P_i)$ is the coefficient ideal of $P_i$ along $E$ in $V$. 
\end{definition}

\begin{remark} Let $P$ be a binomial ideal in $W$. Let $\xi \in W$ be a point. Let $V\subset W$ be a hypersurface of $E$-maximal contact for $P$ at $\xi \in V$. By construction, the junior ideal of $P$ in $V$ can be expressed in terms of binomials. This means it is locally generated by binomials or their powers.  
\end{remark}

\begin{remark} Let $Z$ be the next center to be blown up, defined by this inductive procedure: choice in $W_i$ of a hypersurface of $E$-maximal contact for $P_i$ by means of theorem \ref{gen} and computation of the junior ideal $J_{i-1}$ inside this hypersurface. Then, set $$Z=\cap_{i\in \mathcal{I}}\{x_i=0\} \text{ with } \mathcal{I}\subseteq \{1,\ldots,n\}.$$
\end{remark}
 
\subsection{Definition and properties of the $E$-resolution function} \label{defprop}

\begin{definition} \label{Einv}
Let $(W,(J,c),H,E)$ be a binomial basic object along $E$. For all point $\xi\in \ESing(J,c)$ the $E$-resolution function $t$ will have $n$ components with lexicographical order, and it will be of one of the following types: $$\begin{array}{ll} (a) &
t(\xi)\!=\!(t_n(\xi),t_{n-1}(\xi),\ldots,t_{n-r}(\xi),\ \infty,\
\infty,\ldots,\infty)
\\ (b) & t(\xi)\!=\!(t_n(\xi),t_{n-1}(\xi),\ldots,t_{n-r}(\xi),
\Gamma(\xi),\infty,\ldots,\infty)
\\ (c) & t(\xi)\!=\!(t_n(\xi),t_{n-1}(\xi),\ldots,t_{n-r}(\xi),\ldots
\ldots \ldots,t_1(\xi)) \end{array}\ \text{ with } t_i(\xi)\!=\!\frac{\theta_i}{c_{i+1}} \text{ if } \theta_i\!>\!0 $$ where $\theta_i=\Eord_{\xi}(I_i)$ and $c_{i+1}=\max\ \Eord(P_{i+1})$ is the critical value in dimension $i$.

In the case $J_i=1$, define $t_i(\xi)=\infty$ and complete the $E$-resolution function $t$ with so many $\infty$ com\-ponents as needed in order to have always the same number of components, that is, $(t_{i-1}(\xi),\ldots,t_1(\xi))=(\infty,\ldots,\infty)$. 

If $\theta_i=0$ then $t_i(\xi)=\Gamma(\xi)$, where $\Gamma$ is the resolution function corresponding to the \emph{monomial case}, see \cite{course}. And complete the $E$-resolution function $(t_{i-1}(\xi),\ldots,t_1(\xi))=(\infty\ldots,\infty)$.
\end{definition}

\begin{remark}
To compute the maximal value of the $E$-resolution function $t$, it is enough to look at the points of maximal $E$-order, at each dimension $i$, where $\Etop(P_i)\subset \Etop(I_i)$. \\ Note that in this case $c_{i+1}=\theta_{i+1}$. 
\end{remark}

\begin{remark} \label{localinv}
Let $(W,(J,c),H,E)$ be a binomial basic object along $E$. By construction, the value of the function $t$ at a point $\xi$ of the $E$-singular locus $\ESing(J,c)$ only depends on the point $\xi$. 

Notice that the value of the function $t$ at any point does not depend on the Gr$\ddot{\rm o}$bner basis of the ideal $J$ fixed at the beginning of the $E$-resolution process. This is because the $E$-order of an ideal $I_i$ is independent of the selected set of generators of $I_i$.
\end{remark}

\begin{remark} Let $v_i$ be the number of invertible variables in dimension $i$ at the point $\xi$. 
Note that $v_i$ provides a lower bound $\sum_{i=1}^n v_i$ for the dimension of the next center of blowing up.
\end{remark}

\begin{lemma} Let $J\subset \mathcal{O}_W$ be a binomial ideal as in definition \ref{jota}. Let $\xi\in W$ be a point. Let $\mathcal{I}$ be a totally ordered set with the lexicographical order. The function $$\begin{array}{rl} t: & W \rightarrow (\mathcal{I},\leq) \\ & \ \xi \ \rightarrow t(\xi) \end{array}$$ is upper semi-continuous.
\end{lemma}

\begin{proof} \ 
\begin{enumerate}
 \item[-] Since the function $\Gamma$ is upper semi-continuous (see \cite{course}), its extension to the $E$-singular locus of the pair $(J,c)$ $$\begin{array}{rccl} \Gamma(J): & \ESing(J,c) & \rightarrow & (\mathbb{Z}\times\mathbb{Q}\times \mathbb{Z}^n,\leq) \\ & \ \xi & \rightarrow & \Gamma(J)(\xi):=(-\Gamma_1(\xi),\Gamma_2(\xi),\Gamma_3(\xi)) \end{array}$$ is also an upper semi-continuous function. Note that if the ideal $J$ is a monomial ideal, then $\Eord_{\xi}(J)=\ord_{\xi}(J)$ for all $\xi\in W$ and therefore $\ESing(J,c)=\Sing(J,c)$. 
 
By proposition \ref{usc} the $E$-order is an upper semi-continuous function, therefore each $t_i$ $$\begin{array}{rl} t_i: & W \rightarrow (\mathbb{Q},\leq) \\ & \xi \ \rightarrow t_i(\xi) \end{array}$$ is an upper semi-continuous function for all $1\leq i \leq n$.
 \item[-] By induction, each tuple $(t_n),(t_n,t_{n-1}),(t_n,t_{n-1},t_{n-2}),\ldots,(t_n,\ldots,t_j)$ is an upper semi-continuous function, since every component $t_i$ it is. Hence $t=(t_n,\ldots,t_1)$ is an upper semi-continuous function.
\end{enumerate} \vspace*{-0.4cm}
\end{proof}

\begin{corollary} As a consequence, $$\EMaxB(t)=\{\xi\in \ESing(J,c)|\ t(\xi)=\max\ t\}$$ is a closed set. In fact, it is the next center to be blown up. 
\end{corollary}

\begin{remark}
Moreover, by construction $\EMaxB(t)=Z=\cap_{i\in \mathcal{I}}\{x_i=0\}$ with $\mathcal{I}\subseteq \{1,\ldots,n\}$.
\end{remark}

\subsection{Hypersurfaces of $E$-maximal contact} \label{Ehipmaxc} 

\begin{notation}
We denote as {\it $i$-th chart} the chart where we divide by $x_i$. For example, if the center of the next blow up is the point $\xi=(0,\ldots,0,1,\ldots,1)$ where $\xi_k=0$ for $k=1,\ldots,s$ and $\xi_l=1$ for $l=s+1,\ldots,n$, this transformation is expressed: 
$$\begin{array}{ccc} K[x_1,\ldots,x_s,y_1,\ldots,y_{n-s}]_y & \rightarrow & k[x_1,\ldots,x_s,y_1,\ldots,y_{n-s},\frac{x_1}{x_i},\ldots,\frac{x_{i-1}}{x_i},\frac{x_{i+1}}{x_i},\ldots,\frac{x_s}{x_i}]_y \\ x_i & \rightarrow & \hspace*{-2cm} x_i \\ x_j & \rightarrow & \hspace*{1.3cm} \frac{x_j}{x_i} \text{ for } j\neq i, j=1,\ldots,s \\ y_j & \rightarrow & \hspace*{2cm} y_j \text{ for } j\neq i, j=1,\ldots,n-s \\
\end{array} $$ For simplicity, we will denote each $\frac{x_j}{x_i}$ as $x'_j$. 
\end{notation}

\begin{proposition} \label{dropblup} Let $P\subset \mathcal{O}_W$ be a binomial ideal as in definition \ref{jota}. Let $\xi\in W$ be a point where $\Eord_{\xi}(P)=\theta>0$ is maximal. Let $Z=\cap_{i\in \mathcal{I}}\{x_i=0\}\subset \Etop(P)$ be the next center to be blown up, with  $\mathcal{I}\subseteq \{1,\ldots,n\}$. Let $W \stackrel{\pi}{\leftarrow} W'$ be the blow up along $Z$, where $W'$ is the strict transform of $W$. 

The total transform of $P$ in $W'$ satisfies $P\mathcal{O}_{W'}=I(Y')^{\theta}\cdot P^{\curlyvee}$ where $P^{\curlyvee}$ is the weak transform of $P$, and $Y'$ is the exceptional divisor. Then $$\Eord_{\xi}(P)\geq \Eord_{\xi'}(P^{\curlyvee})$$ for all $\xi'\in W'$ with $\pi(\xi')=\xi$.  
\end{proposition}

\begin{proof} By construction $P$ is generated by binomials of the form $y^{\gamma}x^{\alpha}-bx^{\beta}$ with $|\alpha|> 0$ as in equation (\ref{efegamma}). Let $f(x,y)=y^{\gamma}x^{\alpha}-bx^{\beta}\in P$ be an equation in $P$ such that $\Eord_{\xi}(f)=\theta=|\alpha|$. Let assume $f$ has no common factors. After the blow up along $Z$:

\begin{enumerate}
	\item[-] At some $j$-th chart where $\alpha_j> 0$, the $E$-order of $f^{\curlyvee}$ drops respect to the $E$-order of $f$ $$f^{*}(x',y)={{x_j}'}^{\sum_{i\in \mathcal{I}}\alpha_i}\cdot f^{\curlyvee}(x',y)={{x_j}'}^{\sum_{i\in \mathcal{I}}\alpha_i}\cdot(y^{\gamma}{x'}^{\alpha^{*}}-b{x'}^{\beta}{{x_j}'}^{\sum_{i\in \mathcal{I}}(\beta_i-\alpha_i)})$$ where $\alpha^{*}=\alpha- (0,\ldots,0, \alpha_j,0,\ldots,0)$, $|\alpha^{*}|=|\alpha|-\alpha_j$, and $\sum_{i\in \mathcal{I}}\alpha_i=\theta$ since $\Eord_{\xi}(f)=\theta$ and $Z\subset \Etop(P)$. 

The center $Z$ satisfies $Z\subset \Etop(P)$ then $\sum_{i \in \mathcal{I}} (\beta_i-\alpha_i)\geq 0$, so $$|\beta|+\sum_{i \in \mathcal{I}} (\beta_i-\alpha_i)\geq |\beta|\geq |\alpha|> |\alpha^{*}|$$ and therefore  $\Eord_{\xi'}(f^{\curlyvee})= |\alpha^{*}|< |\alpha|$.           

 \item[-]  At some $(k+j)$-th chart where $1\leq j\leq l$, $\beta_{k+j}>0$,  $$f^{\curlyvee}(x',y)=y^{\gamma}{x'}^{\alpha}-b{x'}^{\beta^{*}}{{x_{k+j}}'}^{\sum_{i\in \mathcal{I}}(\beta_i-\alpha_i)}$$ with $\beta^{*}=\beta-(0,\ldots,0,\beta_{k+j},0,\ldots,0)\in \mathbb{N}^n$. Then $\Eord_{\xi'}(f^{\curlyvee})\leq |\alpha|$. 
\end{enumerate} This provides $\Eord_{\xi'}(P^{\curlyvee})\leq \Eord_{\xi}(P)$.
\end{proof}

\begin{remark} Note that $Z\subset \Etop(P)$ does not imply $Z\subset top(P)$. In the previous proof, it is not allowed to use $\Eord_{\xi}(P)\leq \ord_{\xi}(P)$ and that the property to be proved holds for the order function, because these facts can not ensure $Z\subset top(P)$.  
\end{remark}

\begin{corollary} As a consequence of proposition \ref{dropblup},  $\max\ \Eord(P)\geq \max\ \Eord(P^{\curlyvee}).$
\end{corollary}

\begin{lemma} {\bf Existence of hypersurfaces of $E$-maximal contact} \label{hiper} \\
Let $P_l\subset \mathcal{O}_{W_l}$ be a binomial ideal (\ref{jota}) in dimension $l$, $n\geq l \geq 1$. Suppose $l=n$. \\ There exists a hypersurface of $E$-maximal contact for $P_n$ of type $\{x_i\!=\!0\}$. 
\end{lemma}

\begin{proof} 
By construction $P_n$ is generated by binomial equations $y^{\gamma}x^{\alpha}-bx^{\beta}$ with $|\alpha|> 0$ as in (\ref{efegamma}), or their $p$-th powers. Set $\theta_n=\max\ \Eord(P_n)$. Let $f(x,y)=y^{\gamma}x^{\alpha}-bx^{\beta}\in P_n$ be a binomial equation such that $\Eord_{\xi}(f)=\theta_n=|\alpha|$ at some point $\xi \in W$. Assume $f$ has no common factors. 

By corollary (\ref{esing}), $\ESing(P_n,\theta_n)\subseteq \{x_i=0\}$ for some $i\in \{1,\ldots,k\}$, $\alpha_i>0$. Blow up along some combinatorial center $Z=\cap_{j\in \mathcal{I}}\{x_j=0\}$ $\subset \Etop(P_n)\subset \{x_i=0\}$, with $\mathcal{I}\subseteq \{1,\ldots,n\}$. For some point $\xi'\in W'$ such that $\pi(\xi')=\xi$, if the $E$-order remains constant $\Eord_{\xi'}(P_n^{\curlyvee})=\theta_n$ then $\ESing(P_n^{\curlyvee},\theta_n)\subseteq \{x_i'=0\}$. 

Proceed as in proof of proposition (\ref{dropblup}). At some $j$-th chart with $\alpha_j> 0$ the $E$-order of $f^{\curlyvee}$ strictly decreases respect to the $E$-order of $f$. Then look at some ${(k+j)}$-th chart with $1\leq j\leq l$, $\beta_{k+j}>0$. At this chart $\Eord_{\xi'}(f^{\curlyvee})\leq|\alpha|$.

Since $\Eord_{\xi'}(P_n^{\curlyvee})=|\alpha|=\theta_n$ then $\Eord_{\xi'}(f^{\curlyvee})=|\alpha|$. Moreover  $\alpha_i>0$, then $f^{\curlyvee}$ satisfies the hypothesis of theorem (\ref{gen}). Therefore, $$\ESing(P_n^{\curlyvee},|\alpha|)\subseteq S_{f^{\curlyvee},E}(\xi')\subseteq \{x_i'=0\}. \vspace*{-0.5cm}$$  
\end{proof}	

\begin{lemma} {\bf Stability under blowing-up} \label{esta} \\ 
Let $V$ be a hypersurface of $E$-maximal contact for a binomial ideal $P$. Let $Z\subset \Etop(P)$ be the next combinatorial center to be blown up. Let $V'=\{x_i'=0\}$ be the strict transform of $V$ by the blow up along $Z$.

If the $E$-order remains constant after the blow up, then $V'$ is a hypersurface of $E$-maximal contact for $P^{\curlyvee}$.
\end{lemma}

\begin{proof} With the notation of the previous lemma, note that $f^{\curlyvee}\in P^{\curlyvee}$ is of the form $f^{\curlyvee}(x',y)=y^{\gamma}{x'}^{\alpha'}-b{x'}^{\beta'}$ where $\alpha_i'>0$. It satisfies the conditions of proof of lemma (\ref{hiper}).
\end{proof}

\begin{remark} In the proof of lemma (\ref{hiper}) it is necessary to use $\Eord_{\xi}(P_i)\geq \Eord_{\xi'}(P_i^{\curlyvee})$ for a binomial ideal $P_i$ in dimension $i$, whereas the $E$-resolution function remains constant at higher dimension $(t_n'(\xi'),\ldots,t_{i+1}'(\xi'))=(t_n(\xi),\ldots,t_{i+1}(\xi))$. This fact is consequence of lemma (\ref{bajainv}).
\end{remark}

\begin{parag}
It is necessary to consider $H$ as part of a BBOE in order to take into account the exceptional divisors coming from the previous blow ups and their transforms (which, in fact, belong to $E$).  

When the higher coordinates of the $E$-resolution function remain constant under the blow up, the strict transforms of the hypersurfaces of $E$-maximal contact coming from the previous step are again hypersurfaces of $E$-maximal contact. In this situation, it is indispensable to consider these strict transforms at the same position as their ancestors in the previous step. 

This means to determine the set of \emph{permissible} hypersurfaces. 
\end{parag}

\begin{parag} \label{perhyp} Let $H_i^{(k)}\neq\emptyset$ be the exceptional divisor in dimension $i$, at the $k$-th stage of the $E$-resolution process. 

Let $\xi^{(k_0)}\in W^{(k_0)}, \xi^{(k-1)}\in W^{(k-1)}$ and $\xi^{(k)}\in W^{(k)}$ be points where the $E$-resolution function attains its maximal value and satisfy $\pi_k(\xi^{(k)})=\xi^{(k-1)},\ldots,\pi_{k_0+1}(\xi^{(k_0+1)})=\xi^{(k_0)}$, $k_0\leq k$.  

If $\max\ t_i^{(k)}>0$, let $k_0$ be the smaller superscript such that $$\max\ t_i^{(k_0-1)}>\max\ t_i^{(k_0)}=\ldots=\max\ t_i^{(k)}$$ when the higher coordinates of the $E$-resolution function remain constant under the blow up, {\small $$\begin{array}{rl}(t_n^{(k)}(\xi^{(k)}),t_{n-1}^{(k)}(\xi^{(k)}),\ldots,t_{i}^{(k)}(\xi^{(k)})) & \hspace*{-0.2cm} =(t_n^{(k-1)}(\xi^{(k-1)}),t_{n-1}^{(k-1)}(\xi^{(k-1)}),\ldots,t_{i}^{(k-1)}(\xi^{(k-1)})) \\ & \hspace*{-0.2cm} =\ldots= (t_n^{(k_0)}(\xi^{(k_0)}),t_{n-1}^{(k_0)}(\xi^{(k_0)}),\ldots,t_{i}^{(k_0)}(\xi^{(k_0)})). \end{array}$$} 
\end{parag}

\begin{definition}
Let $H_i^{(k)-}$ be the set of hypersurfaces of $H_i^{(k)}$ which are strict transforms of hypersurfaces in $H_i^{(k_0)}$, following the notation of paragraph \ref{perhyp}. Then, set $H_i^{(k)}$ as the disjoint union $H_i^{(k)}=H_i^{(k)+}\sqcup H_i^{(k)-}$.

The hypersurfaces in $H_i^{(k)-}$ are called \emph{permissible} hypersurfaces, in the sense that they are the only hypersurfaces of $E$-maximal contact that can be considered in practice to make induction on the dimension at this step of the $E$-resolution process.    
\end{definition}

\begin{remark} If $H=\emptyset$ then any hypersurface of $E$-maximal contact is permissible. 
\end{remark}

\subsection{Stability of binomial basic objects along $E$} \label{stabin} 

In this section we prove the stability of the structure of a BBOE, following the same point of view as in \cite{strong} page $837$. 

After a blow up $\pi$, when the coordinates of the $E$-resolution function from dimension $n$ up to dimension $i+1$ remain constant, we construct the corresponding BBOE in dimension $i$.

At some stage of the $E$-resolution process, the ideals $(J_n,\ldots,J_{i+1})$ are defined in $W_n,\ldots,W_{i+1}$ in a neighborhood of the point $\xi$. After a blow up $W_n\stackrel{\pi}{\leftarrow}W_n'$, it is possible to construct   $(J_n',\ldots,J_{i+1}')$ in $W_n',\ldots,W_{i+1}'$ in a neighborhood of the point $\xi'$, with $\pi(\xi')=\xi$.

\begin{notation} For some $j+1$ with $n\geq j\geq i+1$, following the notation of \cite{strong}, denote by 
$$O_{j+1}'\!=\!\{\xi'\in W_{j+1}'|\  (t_n'(\xi'),\ldots,t_{j+2}'(\xi'))\!=\!(t_n(\xi),\ldots,t_{j+2}(\xi)), \theta_{j+1}'\!=\!\theta_{j+1}\}$$ the set of points in $W_{j+1}'$ where the $E$-resolution function remains constant up to dimension $j+2$ and $\theta_{j+1}'=\theta_{j+1}$. 

The set $T_{j+1}'$ is the set of points in $W_{j+1}'$ where the $E$-resolution function remains constant up to dimension $j+1$, $$T_{j+1}'=\{\xi'\in W_{j+1}'\ |\  (t_n'(\xi'),\ldots,t_{j+1}'(\xi'))=(t_n(\xi),\ldots,t_{j+1}(\xi))\}.$$ 
\end{notation}

\begin{proposition} At any stage $j+1$ of the $E$-resolution process $T_{j+1}'=O_{j+1}'$.
\end{proposition}

\begin{proof}
Since $c_k=\Eord(P_k)=\Eord(I_k)=\theta_k$ for all $n\geq k\geq 1$ at the points of maximal $E$-order, the $E$-resolution function satisfies 
$$ (t_n(\xi),\ldots,t_{j+1}(\xi))\!=\! { \left(\frac{\theta_n}{c_{n+1}},\frac{\theta_{n-1}}{c_{n}},\ldots,\frac{\theta_{j+2}}{c_{j+3}},\frac{\theta_{j+1}}{c_{j+2}}\right)=
\left(\frac{\theta_n}{c},\frac{\theta_{n-1}}{\theta_n},\ldots,\frac{\theta_{j+2}}{\theta_{j+3}},\frac{\theta_{j+1}}{\theta_{j+2}}\right)}.$$

It is obvious that $T_{j+1}'\subset O_{j+1}'$. If $\xi'\in O_{j+1}'$ then $$\left(\frac{\theta_{n}'}{c},\frac{\theta_{n-1}'}{\theta_{n}'},\ldots,\frac{\theta_{j+2}'}{\theta_{j+3}'}\right)=\left(\frac{\theta_{n}}{c},\frac{\theta_{n-1}}{\theta_{n}}, \ldots,\frac{\theta_{j+2}}{\theta_{j+3}}\right)$$ what implies $\theta_{j+2}'=\theta_{j+2}$. On the other hand $\theta_{j+1}'=\theta_{j+1}$ so 
$t_{j+1}'(\xi')=t_{j+1}(\xi)$.  Therefore $\xi'\in T_{j+1}'$.
\end{proof}	

\begin{parag} \label{commuta}
After the blowing up $W_n\stackrel{\pi}{\leftarrow}W_n'$, the new setup constructed in $W_n',\ldots,W_{i+1}'$ satisfy, for all $n\geq j\geq i+1$ : 
\begin{itemize}
	\item $W_j'=W_j^{\curlyvee}$ for all point in $T_{j+1}'$, where $W_j^{\curlyvee}$ is the strict transform of $W_j$. 
	\item $J_j'=J_j^{!}$ and $I_j'=I_j^{\curlyvee}$ in $T_{j+1}'$. To construct the ideals $M_j'$ (remark \ref{fac}) the divisors $D_j'$ are defined as in \cite{strong} \begin{equation} \label{leydes} D_j'=\left\{ \begin{array}{ll} D_j^*+(\theta_{j}-c_{j+1})\cdot Y' &
\text{ if } (t_n'(\xi'),\ldots,t_{j+1}'(\xi'))=(t_n(\xi),\ldots,t_{j+1}(\xi))  \\ \emptyset & \text{ otherwise } \end{array} \right.\end{equation} 
in the neighborhood of a point $\xi\in W_j$. Where $D_j^*$ is the pullback of $D_j$ by the blow up $\pi$, $Y'$ is the exceptional divisor, the point $\xi'\in W_j'$ satisfies $\pi(\xi')=\xi$, $\theta_{j}=\Eord_{\xi}(I_j)$ and $c_{j+1}$ is the corresponding critical value. 
	\item $P_j'=P_j^{\curlyvee}$ in $T_{j}'$. The exceptional divisor $H_j'$ is defined as in \cite{strong}
\begin{equation} \label{divex} H_j'=\left\{ \begin{array}{ll}  H_j^{\curlyvee} &
\text{ if }  \xi'\in T_{j}' \\ (Y'+(H_1\cup \ldots \cup H_n)^{\curlyvee})-(H_n'+\cdots
+H_{j+1}') & \text{ otherwise } \end{array} \right.\end{equation} 

in the neighborhood of a point $\xi\in W_j$. Where $H_j^{\curlyvee}$ is the strict transform of $H_j$ by the blow up $\pi$, $Y'$ is the exceptional divisor, the point $\xi'\in W_j'$ satisfies $\pi(\xi')=\xi$, $\theta_{j}'=\Eord_{\xi'}(I_j')$ and $\theta_{j}=\Eord_{\xi}(I_j)$. 	
\item $E_j'=E'\cap W_j'$. 
\end{itemize}
\end{parag}

Assume each $W_j$ is of $E$-maximal contact for $P_{j+1}$ in a neighborhood of $\xi$, for $n-1\geq j\geq i$. Analogously assume $W_j'$ is of $E$-maximal contact for $P_{j+1}'$ in a neighborhood of $\xi'$, for $n-1\geq j\geq i+1$.
\medskip

The key point is find a regular hypersurface $W_i'$ in a neighborhood of $\xi'\in W_{i+1}'$ such that $W_i'$ is of $E$-maximal contact for $P_{i+1}'$ and the previous commutativity relations (\ref{commuta}) are also fulfilled for $j=i$.
\medskip

\begin{parag}{\bf Constructions in dimension $i$.} 
\medskip

By hypothesis $W_i$ is a hypersurface of $E$-maximal contact for $P_{i+1}$. And by construction $P_{i+1}'=P_{i+1}^{\curlyvee}$ in $T_{i+1}'$ where $\theta_{i+1}'=\theta_{i+1}$. Then $W_i^{\curlyvee}$ is a hypersurface of $E$-maximal contact for $P_{i+1}'$. \\

So, we define $W_i'=W_i^{\curlyvee}$ in $T_{i+1}'$. Note that the $E$-order has not remained constant along the points of $W_{i+1}'$ outside $T_{i+1}'$. At those points, by lemma (\ref{hiper}) there exists a hypersurface of $E$-maximal contact for $P_{i+1}'$ in a neighborhood of $\xi'$, for all $\xi'\in \Etop(P_{i+1}')$.
\medskip

Let $J_i'$ be the junior ideal of $P_{i+1}'$ in $W_i'$. 
\begin{itemize}
	\item Assume $P_{i+1}$ and $P_{i+1}'$ are not bold regular or $1$. In $T_{i+1}'$ it holds  $\theta_{i+1}'=\theta_{i+1}$ then $$c_{i+1}'=\Eord(P_{i+1}')=\theta_{i+1}'=\theta_{i+1}=\Eord(P_{i+1})=c_{i+1}$$ at the points of maximal $E$-order. By proposition (\ref{conmutdebil})
\begin{equation} \label{Ecoef} (\ECoeff_{W_i}(P_{i+1}))^{!}=\ECoeff_{W_i'}(P_{i+1}^{\curlyvee}) \end{equation} 
therefore, the controlled transform of $J_i$ satisfies  $$(J_i)^{!}=(\ECoeff_{W_i}(P_{i+1}))^{!}=\ECoeff_{W_i'}(P_{i+1}')=J_i'$$ since $P_{i+1}'=P_{i+1}^{\curlyvee}$.	
	\item If $P_{i+1}'$ is bold regular or $1$, the previous equality (\ref{Ecoef}) gives $(\ECoeff_{W_i}(P_{i+1}))^{!}=0$. Then, by proposition \ref{boldreg}, $P_{i+1}$ is bold regular or $1$ and $J_i'=J_i^{!}=1$.
  \item If $P_{i+1}=1$ then $P_{i+1}'=P_{i+1}^{\curlyvee}=1$ in $T_{i+1}'$ and $J_i'=J_i^{!}=1$.
  \item If $P_{i+1}\neq 1$ and it is bold regular, then $P_{i+1}=<y^{\gamma}(1-\mu y^{\delta})x_1^{\alpha_1}>$ where $\mu \in K$, $\gamma,\delta\in \mathbb{Z}^n$, $\alpha_1\in \mathbb{N}$, $\alpha_1=\max\ \Eord(P_{i+1})>0$.  The hypersurface $\{x_1=0\}$ is a hypersurface of $E$-maximal contact for $P_{i+1}$. A blow up along a combinatorial center $Z\subset \{x_1=0\}$ provides $P_{i+1}'=P_{i+1}^{\curlyvee}=<y^{\gamma}(1-\mu y^{\delta})>$. So  $\Eord_{\xi'}(P_{i+1}')=0$ for all $\xi'\in W_i'$ and then $\theta_{i+1}'=0<\theta_{i+1}$ therefore $\xi'\not\in T_{i+1}'$.  
\end{itemize}
Hence $J_i'=J_i^{!}$ in $T_{i+1}'$.
\medskip

The ideal $M_i'=I_{W_i'}(D_i'\cap W_i')$ has support in the divisor $D_i'$, defined as in equation (\ref{leydes}). 
If $J_i=M_i\cdot I_i$ the controlled transform $J_i^{!}=I_{W_i'}(Y'\cap W_i')^{-c_{i+1}}\cdot J_i^{*}$ in $W_i'$. 
In addition if $M_i'=I_{W_i'}(Y'\cap W_i')^{\theta_i-c_{i+1}}\cdot M_i^{*}$ in $T_{i+1}'$ with $\theta_i=\Eord_{\xi}(I_i)$ and $I_i^{\curlyvee}=I_{W_i'}(Y'\cap W_i')^{-\theta_i}\cdot I_i^{*}$, then $$\begin{array}{rl} J_i' & \hspace*{-0.2cm} =I_{W_i'}(Y'\cap W_i')^{-c_{i+1}}\cdot J_i^{*}=I_{W_i'}(Y'\cap W_i')^{-c_{i+1}}\cdot (M_i\cdot I_i)^{*} \\ & \hspace*{-0.2cm} =I_{W_i'}(Y'\cap W_i')^{-c_{i+1}}\cdot M_i^{*}\cdot I_i^{*} \\ & \hspace*{-0.2cm} =I_{W_i'}(Y'\cap W_i')^{-c_{i+1}}\cdot I_{W_i'}(Y'\cap W_i')^{c_{i+1}-\theta_i}\cdot M_i'\cdot I_{W_i'}(Y'\cap W_i')^{\theta_i}\cdot I_i^{\curlyvee}\!=\!M_i'\cdot I_i^{\curlyvee} \end{array}$$ so $J_i'=M_i'\cdot I_i'$ where $I_i'=I_i^{\curlyvee}$ in $T_{i+1}'$. Note that $J_i'=I_i'$ outside $T_{i+1}'$, since in this case $D_i'=\emptyset$. 
\medskip

Let $\xi'\in T_{i+1}'$ be a point such that $\pi(\xi')=\xi$. We have $I_i'=I_i^{\curlyvee}$ in $T_{i+1}'$. If $\theta_i'=\Eord_{\xi'}(I_i')=\Eord_{\xi}(I_i)=\theta_i$ in a neighborhood of $\xi'$ then $\xi'\in T_{i}'$.  
\medskip

Construct now the companion ideal $P_i'$ in $T_{i}'$. We have $M_i'=M_i^*\cdot I_{W_i'}(Y'\cap W_i')^{\theta_i-c_{i+1}}$, $c_{i+1}'=c_{i+1}$ and $\theta_i'=\theta_i$, then $$\begin{array}{rl} P_i^{\curlyvee} & =(I_i + M_i^{\frac{\theta_i}{c_{i+1}-\theta_i}})^{\curlyvee}=I_i^{\curlyvee} + (M_i^{\frac{\theta_i}{c_{i+1}-\theta_i}})'= I_i^{\curlyvee} + (M_i')^{\frac{\theta_i}{c_{i+1}-\theta_i}}= \\ & =I_i' + (M_i')^{\frac{\theta_i'}{c_{i+1}'-\theta_i'}}=P_i' \end{array}$$ because of $I_i^{\curlyvee}=I_i'$. The case $P_i=I_i$ is obvious. Therefore $P_i'=P_i^{\curlyvee}$ in $T_{i}'$.

Finally, the exceptional divisor $H_i'$ is defined as in equation (\ref{divex}), and $E_i'=E\cap W_i'$. 
\end{parag}

\begin{remark} \label{pprima}
As a consequence, all the results of section (\ref{Ehipmaxc}) hold replacing $P^{\curlyvee}$ by $P'$. 
\end{remark}

\subsection{Commutativity} \label{comm}

The computation of the coefficient ideal along $E$ commutes with the blowing up.

\begin{corollary} \label{transfcoef} Let $P$ be an ideal in $W$. Let $\xi\in W$ be a point where $c=\Eord_{\xi}(P)$.  
Let $\pi$ be the blow up along $Z\subset \Etop(P)$, where $\pi(\xi')=\xi$ and $c'=\Eord_{\xi'}(P')$. If $c'=c$ then the coefficient ideal of $P$ along $E$ with respect to $V$ satisfies
$$\ECoeff_{V'}(P')=(\ECoeff_V(P))'$$ where $P'=P^{\curlyvee}$ and $V'$ is the strict transform of $V$.
\end{corollary}

\begin{proof} By proposition \ref{conmutdebil} $$\ECoeff_{V'}(P')=\ECoeff_{V^{\curlyvee}}(P^{\curlyvee})=(\ECoeff_V(P))^{!}=
(\ECoeff_V(P))'. \vspace*{-0.5cm}$$
\end{proof}

The inductive step from dimension $i$ to dimension $i-1$ commutes with the blowing up.   

\begin{corollary} \label{conmut-1} 
Let $(W,(J,c),H,E)$ be a binomial basic object along $E$. By induction on the dimension, construct the binomial basic objects $$(W,(J,c),H,E)=(W_n,(J_n,c_{n+1}),H_n,E_n),\ldots,(W_i,(J_i,c_{i+1}),H_i,E_i).$$ Let $$(W_n',(J_n',c_{n+1}),H_n',E_n'),\ldots,(W_i',(J_i',c_{i+1}'),H_i',E_i')$$ be their transforms by the blow up $\pi$ along a permissible center $Z$. Let $\xi\in Z$ be a point. In the neighborhood of a point $\xi'\in W_i'$ such that $\pi(\xi')=\xi$, if the $E$-resolution function remains constant  $(t_n'(\xi'),t_{n-1}'(\xi'),\ldots,t_i'(\xi'))=(t_n(\xi),t_{n-1}(\xi),\ldots,t_i(\xi))$ then $$J_{i-1}'=(J_{i-1})'.$$
\end{corollary}

\begin{proof}
Let $V=\{x_j=0\}$ be a hypersurface of $E$-maximal contact for $P_i$ in a neighborhood of $\xi$.
If $(t_n'(\xi'),t_{n-1}'(\xi'),\ldots,t_i'(\xi'))=(t_n(\xi),t_{n-1}(\xi),\ldots,t_i(\xi))$ then $c_{i}'=c_{i}$ and $\ESing(P_i',c_i')\subseteq V'$. By definition $J_{i-1}=\ECoeff_V(P_i)$. By corollary \ref{transfcoef} $\ECoeff_{V'}(P_i')=(\ECoeff_V(P_i))'$, then 
$$J_{i-1}'=\ECoeff_{V'}(P_i')=(\ECoeff_V(P_i))'=(J_{i-1})'. \vspace*{-0.5cm}$$
\end{proof}

\subsection{Decrease of the $E$-resolution function} \label{decres}

In this section we prove that the $E$-resolution function drops lexicographically after blowing up.

\begin{lemma} \label{bajainv} Let $(W,(J,c),H,E)$ be a binomial basic object along $E$, where $J\neq 1$. Let $W \stackrel{\pi}{\leftarrow}W'$ be the blow up along $Z=\EMaxB(t)$. Then $$t(\xi)>t'(\xi')$$ for all $\xi\in Z$, $\xi'\in Y'=\pi^{-1}(Z)$, $\pi(\xi')=\xi$, where $t$ is the $E$-resolution function corresponding to $(W,(J,c),H,E)$ and $t'$ corresponds to $(W',(J',c),H',E')$, its transform by the blow up $\pi$.  
\end{lemma}

\begin{proof} \ 
\begin{itemize}
	\item[\underline{Step 1}]: The $E$-resolution function satisfies $t(\xi)\geq t'(\xi')$.
\medskip

Let $Z=\cap_{i\in \mathcal{I}}\{x_i=0\}$ be the next center of blowing up, with $\mathcal{I}\subseteq \{1,\ldots,n\}$. 

By definition of the $E$-resolution function $t_n(\xi)=\frac{\Eord_{\xi}(I_n)}{c}$ and $t_n'(\xi')=\frac{\Eord_{\xi'}(I_n')}{c}$.  

Proposition \ref{dropblup} and remark \ref{pprima} imply $\Eord_{\xi}(I_n)\geq \Eord_{\xi'}(I_n')$ therefore $t_n(\xi)\geq t_n'(\xi')$.
\begin{itemize}
  \item If $t_n(\xi)>t_n'(\xi')$ then $t(\xi)\geq t'(\xi')$, in fact $t(\xi)>t'(\xi')$.
	\item If $t_n(\xi)=t_n'(\xi')$ then $I_{n-1}'=I_{n-1}^{\curlyvee}$. So $t_{n-1}(\xi)\geq t_{n-1}'(\xi')$ by proposition \ref{dropblup}.  
\end{itemize} 
Hence $(t_n(\xi),t_{n-1}(\xi))\geq (t_n'(\xi'),t_{n-1}'(\xi'))$. By induction on the dimension, if $$(t_n(\xi),t_{n-1}(\xi),\ldots,t_{j+1}(\xi))=(t_n'(\xi'),t_{n-1}'(\xi'),\ldots,t_{j+1}'(\xi'))$$ then $t_j(\xi)\geq t_j'(\xi')$ so $t(\xi)\geq t'(\xi')$.

\begin{itemize}
	\item Non monomial case: By construction the $E$-resolution function satisfies $$(t_n(\xi),t_{n-1}(\xi),\ldots,t_i(\xi))=\left(\frac{\theta_{n}}{c_{n+1}},\frac{\theta_{n-1}}{c_n},\frac{\theta_{n-2}}{c_{n-1}},\ldots,\frac{\theta_{i}}{c_{i+1}}\right)$$ where $c_{j}=\theta_{j}$, for $n\geq j\geq i+1$.
	
If $(t_n'(\xi'),t_{n-1}'(\xi'),\ldots,t_i'(\xi'))=(t_n(\xi),t_{n-1}(\xi),\ldots,t_i(\xi))$ then  $c_{i}=\theta_{i}=\theta_{i}'=c_i'$ whereas $\theta_{i}\neq 0$. So it is enough to show $\theta_{i-1}\geq \theta_{i-1}'$ to obtain $t_{i-1}(\xi)\geq t_{i-1}'(\xi')$.	
	\item Monomial case: If $(t_n'(\xi'),t_{n-1}'(\xi'),\ldots,t_i'(\xi'))=(t_n(\xi),t_{n-1}(\xi),\ldots,t_i(\xi))$ and $\theta_{i-1}=0$ then $t_{i-1}(\xi)=\Gamma(\xi)$. Therefore $t_{i-1}(\xi)\geq t_{i-1}'(\xi')$ by well known properties of $\Gamma$ function.  
\end{itemize}

\item[\underline{Step 2}]: Induction on the dimension of the ambient space $W$. 
\par
Assume $n=1$.  Let $J=I=<x_1^{\theta}(1-x_1^{\alpha-\theta})>$ be a binomial ideal in one variable. 
\begin{itemize}
 \item If $\xi_1=0$ then in a neighborhood of $\xi$, $\Eord_{\xi}(I)=\theta$. The controlled transform of $J$ by $\pi$ is  $J'=<(x_1^{\theta-c})\cdot(1-x_1^{\alpha-\theta})>$ therefore $\Eord_{\xi'}(I')=0<\theta$.
 \item If $\xi_1\neq 0$ then in a neighborhood of $\xi$, $\Eord_{\xi}(I)=0$ so we use $\Gamma$ function. 
 
This case can not come from an ideal in higher dimension. If $J=\ECoeff_{\{x_2=0\}}(P_2)$ then  $P_2=<x_2^{\theta_2}x_1^{\theta}(1-x_1^{\alpha-\theta}),f_2,\ldots,f_r>$ where the variable $x_2$ appears in $f_i$, for $2\leq i\leq r$, with an exponent bigger than or equal to $\Eord_{\xi}(P_2)$. Since $\xi_1\neq 0$ we have  $\Eord_{\xi}(P_2)=\Eord_{\xi}(x_2^{\theta_2})=\theta_2$, but in this case $\ECoeff_{\{x_2=0\}}(P_2)=0$, contradiction.   
\end{itemize}

If $J=<x^{\theta_1}>$ it is enough to note $J'=<(x^{\theta_1-c})>=M'$ with $I'=1$. So $\Eord_{\xi}(I)=\theta_1> \Eord_{\xi'}(I')=0$. 

\par	
Study the case $n>1$. If $(t_n'(\xi'),t_{n-1}'(\xi'),\ldots,t_{j+1}'(\xi'))=(t_n(\xi),t_{n-1}(\xi),\ldots,t_{j+1}(\xi))$ then $t_j(\xi)\geq t_j'(\xi')$ by step $1$. Therefore	
\begin{itemize}
	\item either $\theta_{j}>\theta_{j}'$
	\item or $\theta_{j}=\theta_{j}'$. Observe that this is not possible when $j=1$, hence $t(\xi)>t'(\xi')$.
\end{itemize}
\end{itemize} \vspace*{-0.5cm}
\end{proof}

\section{Algorithm of $E$-resolution of BBOE} \label{algresbin} 

In this section we construct an algorithm of $E$-resolution of binomial basic objects along $E$. 

\subsection{$E$-resolution of BBOE: Inductive step} \label{resind}

A $E$-resolution of binomial basic objects along $E$ in dimension $n-1$ provides a $E$-resolution of binomial basic objects along $E$ in dimension $n$. 

\begin{proposition} \label{saltodim} Let $(W^{(0)},(J^{(0)},c),H^{(0)},E^{(0)})$ be a binomial basic object along $E^{(0)}$ of dimension $n$. Let $$(W^{(0)},(J^{(0)},c),H^{(0)},E^{(0)})\stackrel{\pi_1}{\longleftarrow}\ldots\stackrel{\pi_r}{\longleftarrow}(W^{(r)},(J^{(r)},c),H^{(r)},E^{(r)})\stackrel{\pi_{r+1}}{\longleftarrow}\ldots$$ be an equivariant sequence of transformations of binomial basic objects along combinatorial centers $Z^{(k)}\subset \ESing(J^{(k)},c)$ for $k\geq 0$. 

Let $\mathcal{I}_n$ be a totally ordered set. Let $t^{(k)}: \ESing(J^{(k)},c)\rightarrow \mathcal{I}_n$ be an equivariant function defined as in \ref{Einv} such that $Z^{(k)}=\EMaxB t^{(k)}$, for $k\geq 0$. 

Let $r_0<r$ be the last superscript where $t_n$ has dropped. This means there exists a sequence of transformations  \begin{equation} \label{suc}
\ldots\stackrel{\pi_{r_0}}{\longleftarrow} (W_{n}^{(r_0)},(J_{n}^{(r_0)},c),H_{n}^{(r_0)},E_{n}^{(r_0)})\stackrel{\pi_{r_0+1}}{\longleftarrow}\ldots \stackrel{\pi_r}{\longleftarrow}(W_{n}^{(r)},(J_{n}^{(r)},c),H_{n}^{(r)},E_{n}^{(r)})
\end{equation} 
such that $\EMaxB t_n^{(k)}=\ESing(J_{n-1}^{(k)},c_{n}^{(k)})$ for $r_0\leq k\leq r$, satisfying
$$max\ t_n^{(r_0-1)}>max\ t_n^{(r_0)}=max\ t_n^{(r_0+1)}=\ldots=max\ t_n^{(r)}.$$  
They are equivalent: 
\begin{enumerate}
	\item $max\ t_n^{(r)}>max\ t_n^{(r+1)}$.
	\item $\ESing((J_{n-1}^{(r)})',c_{n}^{(r_0)})=\emptyset$.
\end{enumerate}
\end{proposition}

\begin{proof} By induction on the dimension the sequence (\ref{suc}) of transformations of BBOE in dimension $n$ induces a sequence of transformations of BBOE in dimension $n-1$ $$(W_{n-1}^{(r_0)},(J_{n-1}^{(r_0)},c_n^{(r_0)}),H_{n-1}^{(r_0)},E_{n-1}^{(r_0)})\stackrel{\pi_{r_0+1}}{\longleftarrow}\ldots \stackrel{\pi_r}{\longleftarrow} (W_{n-1}^{(r)},(J_{n-1}^{(r)},c_n^{(r)}),H_{n-1}^{(r)},E_{n-1}^{(r)})$$
Since $max\ t_n^{(r_0)}=max\ t_n^{(r_0+1)}=\ldots=max\ t_n^{(r)}$ we have $c_n^{(r_0)}=c_n^{(r_0+1)}=\ldots=c_n^{(r)}$. By corollary \ref{conmut-1} we obtain $J_{n-1}^{(r_0+1)}=(J_{n-1}^{(r_0)})',\ldots,J_{n-1}^{(r)}=(J_{n-1}^{(r-1)})'$.

On the other hand, lemma \ref{bajainv} implies $$max\ (t_{n-1}^{(r_0)},\ldots,t_{1}^{(r_0)})>max\ (t_{n-1}^{(r_0+1)},\ldots,t_{1}^{(r_0+1)})>\ldots>max\ (t_{n-1}^{(r)},\ldots,t_{1}^{(r)}).$$

\begin{itemize}
\item Assume $max\ t_n^{(r)}>max\ t_n^{(r+1)}$ then $c_n^{(r_0)}=c_n^{(r)}>c_n^{(r+1)}$, where $$c_n^{(r+1)}=max\ \Eord(P_n^{(r+1)})=max\ \Eord(I_n^{(r+1)}).$$ Hence $J_{n-1}^{(r+1)}=\ECoeff_{W_{n-1}^{(r+1)}}(P_n^{(r+1)})\neq (J_{n-1}^{(r)})'$. Therefore $$\hspace*{-1cm} \ESing((J_{n-1}^{(r)})',c_{n}^{(r_0)})=\{\xi\in \ESing((J_{n-1}^{(r+1)}),c_{n}^{(r+1)})|\ \Eord_{\xi}((J_{n-1}^{(r)})')\geq c_n^{(r_0)}\}$$ $$\hspace*{1.6cm}=\{\xi\in \ESing((J_{n-1}^{(r+1)}),c_{n}^{(r+1)})|\ c_{n}^{(r+1)}\geq c_n^{(r_0)}\}=\emptyset.$$ 
	
\item Conversely, suppose $max\ t_n^{(r)}=max\ t_n^{(r+1)}$ then $c_n^{(r_0)}=c_n^{(r)}=c_n^{(r+1)}$ and corollary  \ref{conmut-1} implies $J_{n-1}^{(r+1)}=(J_{n-1}^{(r)})'$.
	
By construction, if $P_n^{(r+1)}=I_n^{(r+1)}+(M_n^{(r+1)})^{\frac{c_n^{(r+1)}}{c_{n+1}^{(r+1)}-c_n^{(r+1)}}}$ with $c_{n+1}^{(r+1)}=c$, then {\small \begin{equation} \hspace*{-1cm} \ESing(P_n^{(r+1)},c_n^{(r+1)})=\ESing(I_n^{(r+1)},c_n^{(r+1)})\cap \ESing(M_n^{(r+1)},c_{n+1}^{(r+1)}-c_n^{(r+1)}) \label{formP} \end{equation}} where $c_n^{(r+1)}=max\ \Eord(P_n^{(r+1)})=max\ \Eord(I_n^{(r+1)})$.

In fact {\small	
\begin{equation} \ESing(P_n^{(r+1)},c_n^{(r+1)})=\ESing(I_n^{(r+1)},c_n^{(r+1)})\cap \ESing(J_n^{(r+1)},c_{n+1}^{(r+1)}). \label{singP} \end{equation}}	

\begin{itemize}
	\item[] One direction is obvious by equality (\ref{formP}). For the other direction, let $\xi\in W^{(r+1)}$ be a point such that $\Eord_{\xi}(J_n^{(r+1)})\geq c_{n+1}^{(r+1)}$ and $\Eord_{\xi}(I_n^{(r+1)})=c_n^{(r+1)}$ then $$\Eord_{\xi}(M_n^{(r+1)})=\Eord_{\xi}(J_n^{(r+1)})-\Eord_{\xi}(I_n^{(r+1)})\geq c_{n+1}^{(r+1)}-c_n^{(r+1)}.$$ So $\xi\in \ESing(P_n^{(r+1)},c_n^{(r+1)})$ because of (\ref{formP}).
\end{itemize}
As a consequence of (\ref{singP}) it holds	$$\hspace*{-1cm} \begin{array}{rl} \EMaxB t_n^{(r+1)} & \hspace*{-0.2cm} =\{\xi\in \ESing(J_n^{(r+1)},c_{n+1}^{(r+1)})|\ t_n^{(r+1)}(\xi)\!=\!max\ t_n^{(r+1)}\} \\ & \hspace*{-0.2cm} =\{\xi\in \ESing(J_n^{(r+1)},c_{n+1}^{(r+1)})|\ \Eord_{\xi}(I_n^{(r+1)})\!=\!max\ \Eord(I_n^{(r+1)})\}\\ & \hspace*{-0.2cm} =\ESing(P_n^{(r+1)},c_n^{(r+1)}).\end{array}$$

By proposition \ref{topes} $$\Etop(P_n^{(r+1)})=\Etop(\ECoeff_{W_{n-1}^{(r+1)}}(P_n^{(r+1)}),c_n^{(r+1)})= \Etop(J_{n-1}^{(r+1)},c_n^{(r+1)})$$ then $$\hspace*{-1cm} \EMaxB t_n^{(r+1)}=\Etop(P_n^{(r+1)})=\ESing(J_{n-1}^{(r+1)},c_{n}^{(r+1)})=\ESing((J_{n-1}^{(r)})',c_{n}^{(r+1)})$$ contradiction with $\ESing((J_{n-1}^{(r)})',c_{n}^{(r_0)})=\emptyset$.
\end{itemize}
\end{proof}

\subsection{$E$-resolution of BBOE: Algorithm} \label{algbin} 

\begin{definition} Let $(W,(P,c),H,E)$ be a binomial basic object along $E$. Denote $$\EMaxB(P)=\{\xi\in W|\ \Eord_{\xi}(P)=max\ \Eord(P)\}$$ the set of points where the ideal $P$ attains its maximal $E$-order. 
\end{definition}

\begin{remark} Note that $\EMaxB(P)=S_{P,E}(\xi)$ for $\xi\in W$ a point of maximal $E$-order. 
\end{remark}

The special treatment of hyperbolic equations forces us to specify how to make induction on the dimension on the ambient space to construct the $E$-resolution function $t$.

\begin{algorithm} {\bf Induction on the dimension.}
\medskip 

Let $(W,(J,c),H,E)$ be a binomial basic object along $E$, $W=W_n$, $dim(W_n)=n$. 
If $\ESing(J,c)\neq \emptyset$, assume $(t_n,\ldots,t_{i+1})$ and ideals $J_n,\ldots,J_{i+1},J_i$ are already given and construct $t_i$ and $J_{i-1}$. 

From $i=n$ to $i=1$: 
\begin{enumerate}
	\item Factorize $J_i=M_i\cdot I_i$ where $M_i$ has support in $D_i$. 
\item Compute $max\ \Eord(I_i)$. Let $\xi\in \Etop(I_i)$ be a point
\begin{itemize}
	\item If $\Eord_{\xi}(I_i)>0$, $t_i(\xi)=\frac{\Eord_{\xi}(I_i)}{\Eord_{\xi}(P_{i+1})}$. Compute the companion ideal $P_i$ and go to step $3$.
	\item If $\Eord_{\xi}(I_i)=0$ for all $\xi\in W_i$ then apply $\Gamma$ function to $M_i$,
	$(t_i(\xi),t_{i-1}(\xi),\ldots,t_1(\xi))=(\Gamma(\xi),\infty,\ldots,\infty)$. Stop.	
\end{itemize}
	\item Fix $c_{i}=max\ \Eord(P_i)$. Let $f\in P_i$ be a binomial such that $\Eord_{\xi}(f)=c_{i}$, $\xi\in W_i$. Corollary \ref{esing} provides a hypersurface $V$ of $E$-maximal contact for $P_i$ in a neighborhood of $\xi$. In fact,  $V=\{x_j=0\}$ for some $1\leq j\leq n$. 
	\item Compute $\ECoeff(P_i)$ respect to $V\in H_i^{-}$, to construct $J_{i-1}$ in $W_{i-1}$ where $dim(W_{i-1})=i-1$.
\begin{itemize}
	\item	If $\ECoeff_V(P_i)=0$ then $P_i$ is bold regular, $J_{i-1}=1$, $(t_{i-1}(\xi),\ldots,t_1(\xi))=(\infty,\ldots,\infty)$. Stop. 
	\item If $\ECoeff_V(P_i)\neq 0$ then $J_{i-1}=\ECoeff_V(P_i)$ and go to step $1$.
\end{itemize}	
\end{enumerate}
\end{algorithm}

\begin{remark}
If $\Eord_{\xi}(I)=0$ for all $\xi\in W$ but there exists at least a point $a\in W$ with $ord_a(I)\neq 0$ then there are hyperbolic equations in $I$.   
\begin{itemize}
	\item Note that in dimension $<n$, this kind of ideals $I$ are considered $I=1$. In step (2) of the algorithm, $\Gamma$ function is applied since $\Eord_{\xi}(J)=\Eord_{\xi}(M)$. With respect to the $E$-order, $J$ behaves as a monomial ideal. 

\item In dimension $n$, apply $\Gamma$ function to $J_n$. Stop when $\ESing(J_n,c)=\emptyset$. 
\end{itemize}
\end{remark}

\begin{theorem} {\bf $E$-resolution of binomial basic objects along $E$.} \label{resobb}
\medskip

An algorithm of $E$-resolution of binomial basic objects of dimension $n$ along a normal crossing divisor $E$ consist of:

\begin{itemize}
	\item[A)] A totally ordered set $(\mathcal{I}_n,\leq)$.
	\item[B)] For each BBOE $(W^{(0)},(J^{(0)},c),H^{(0)},E^{(0)})$ where $dim(W^{(0)})=n$, $J^{(0)}=M^{(0)}\cdot I^{(0)}$ and the ideal $I^{(0)}$ does not contain hyperbolic equations:
\begin{enumerate}
	\item Define an equivariant function $t^{(0)}: \ESing(J^{(0)},c)\rightarrow \mathcal{I}_n$ such that $$\EMaxB t^{(0)}\subset \ESing(J^{(0)},c)$$ is a permissible center for $(W^{(0)},(J^{(0)},c),H^{(0)},E^{(0)})$.  
	\item By induction, assume there exists an equivariant sequence of transformations of BBOE 
\begin{eqnarray}	
(W^{(0)},(J^{(0)},c),H^{(0)},E^{(0)}) \stackrel{\pi_1}{\longleftarrow}\ldots \hspace*{7cm}
\nonumber \\  \ldots\stackrel{\pi_{r-1}}{\longleftarrow} (W^{(r-1)},(J^{(r-1)},c),H^{(r-1)},E^{(r-1)})\stackrel{\pi_r}{\longleftarrow}(W^{(r)},(J^{(r)},c),H^{(r)},E^{(r)})
\label{seq}
\end{eqnarray}
along centers $Z^{(k)}\subset \ESing(J^{(k)},c)$ for $0\leq k\leq r-1$; and equivariant functions $$t^{(k)}: \ESing(J^{(k)},c)\rightarrow \mathcal{I}_n$$ for $0\leq k\leq r-1$, such that $Z^{(k)}=\EMaxB t^{(k)}$. 
	
If $\ESing(J^{(r)},c)\neq \emptyset$ this sequence of transformations can be extended. This means at the $r$-th stage of the $E$-resolution process an equivariant function can be defined $$t^{(r)}: \ESing(J^{(r)},c)\rightarrow \mathcal{I}_n$$ such that $Z^{(r)}=\EMaxB t^{(r)}$ is a permissible center for $(W^{(r)},(J^{(r)},c),H^{(r)},E^{(r)})$.
\end{enumerate}
	\item[C)] For some $r$, the previous sequence of transformations (\ref{seq}) is a $E$-resolution of the original BBOE $(W^{(0)},(J^{(0)},c),H^{(0)},E^{(0)})$, that is, $\ESing(J^{(r)},c)=\emptyset$.
\end{itemize}
\end{theorem}

\begin{remark}
The proof follows the same structure as the proof of the existence of an algorithm of resolution of basic objects (over fields of characteristic zero) given in \cite{course}, page $206$. 
\end{remark}

\begin{proof} By induction on $n$. 
\begin{itemize}
	\item  If $n=1$, fix $\mathcal{I}_1=\mathbb{Q}\sqcup(\mathbb{Z}\times\mathbb{Q}\times\mathbb{Z}^{\mathbb{N}})\sqcup\{\infty\}$. Given the BBOE $(W^{(0)},(J^{(0)},c),H^{(0)},E^{(0)})$ with $dim(W^{(0)})=1$ and $E^{(0)}=\{V(x_1)\}$ define $$\begin{array}{cccc} t^{(0)}: &  \ESing(J^{(0)},c) & \rightarrow & \mathcal{I}_1 \\ & \xi & \rightarrow & \frac{\Eord_{\xi}(I^{(0)})}{c} \end{array}$$ as in \ref{Einv}. The ideal $J^{(0)}$ is a binomial ideal in one variable, then it is generated by binomial equations of the type  $x_1^{\alpha}-bx_1^{\beta}=x_1^{\alpha}\cdot(1-bx_1^{\beta-\alpha})$ where $\alpha,\beta\in\mathbb{N}$ with $\alpha\leq\beta$, and eventually, by monomials $x_1^{\eta}$, $\eta\in\mathbb{N}$. 

By corollary \ref{esing}, $\ESing(J^{(0)},c)\subseteq \{x_1=0\}$, since $I^{(0)}$ does not contain hyperbolic equations. But $\{x_1=0\}$ is a hypersurface in a space of dimension $1$, so it is a closed point. Then 
	\begin{itemize}
	\item either $\ESing(J^{(0)},c)=\emptyset$, the $E$-resolution process stops, 
	\item or $\ESing(J^{(0)},c)=\{x_1=0\}=\EMaxB t^{(0)}$ is the next center to be blown up.
  \end{itemize}
In this case, construct the following transformation of BBOE 	$$(W^{(0)},(J^{(0)},c),H^{(0)},E^{(0)})\stackrel{\pi_1}{\longleftarrow}(W^{(1)},(J^{(1)},c),H^{(1)},E^{(1)})$$ where $J^{(1)}=x_1^{\theta-c}\cdot I^{(1)}$, $\theta=\Eord_{\xi}(I^{(0)})>0$, for $\xi\in \{x_1=0\}$. 

By lemma \ref{bajainv}, $t^{(0)}(\xi)>t^{(1)}(\xi^{(1)})$ for $\xi\in \EMaxB t^{(0)}$, $\pi_1(\xi^{(1)})=\xi$. Hence $max\ t^{(0)}\!>\!max\ t^{(1)}$. 

If fact, since $J^{(0)}$ is a binomial ideal in one variable, $\Eord_{\xi^{(1)}}(I^{(1)})=0$ and in a neighborhood of $\xi^{(1)}$, it holds $(I^{(1)})_{\xi^{(1)}}=1$. 

By induction, assume it has been defined a sequence of transformations at permissible centers, which is a sequence of blow ups along a point, 
\begin{equation} \label{resseq}
(W^{(0)},(J^{(0)},c),H^{(0)},E^{(0)})\stackrel{\pi_1}{\longleftarrow}\ldots \stackrel{\pi_r}{\longleftarrow}(W^{(r)},(J^{(r)},c),H^{(r)},E^{(r)}).
\end{equation} 
Each center is $\EMaxB(t^{(i)})$, $0\leq i\leq r-1$, and the $E$-resolution function $t^{(i)}$ defined after each blow up $\pi_i$ is
\begin{equation}
t^{(i)}(\xi^{(i)})=\left\{\begin{array}{lc} \frac{\Eord_{\xi^{(i)}}(I^{(i)})}{c} & \text{ if } \Eord_{\xi^{(i)}}(I^{(i)})>0 \vspace*{0.2cm}  \\ \Gamma^{(i)}(\xi^{(i)}) & \text{ if } \Eord_{\xi^{(i)}}(I^{(i)})=0 \end{array}\right. \label{defseq}
\end{equation} for $0\leq i\leq r-1$. If $\ESing(J^{(r)},c)\neq \emptyset$ define $t^{(r)}$ as above (\ref{defseq}). 

In this case $\Eord_{\xi^{(i)}}(I^{(i)})=0$ for all $\xi^{(i)}\in W^{(i)}$, $i\geq 1$. Hence apply $\Gamma$ function to the ideals $I^{(i)}$. Lemma \ref{bajainv} implies $$max\ t^{(0)}>max\ t^{(1)}>\ldots> max\ t^{(r)}.$$ Because of the upper semi-continuity of each $t^{(i)}$, there exists an index $N$ such that the sequence (\ref{resseq}) is a $E$-resolution of $(W^{(0)},(J^{(0)},c),H^{(0)},E^{(0)})$. 

Considering that the $E$-order is an equivariant function (\ref{equivdef}), each $t^{(i)}$ is an equivariant function. 

	\item Fix $n>1$. By induction hypothesis there exists a totally ordered set $(\mathcal{I}_{n-1},\leq)$. In dimension $n$, consider $$(\mathcal{I}_n,\leq)=((\mathbb{Q}\sqcup(\mathbb{Z}\times\mathbb{Q}\times\mathbb{Z}^{\mathbb{N}}) \sqcup\{\infty\})\times\mathcal{I}_{n-1},\leq)$$ with lexicographical order. Let  $B^{(0)}=(W^{(0)},(J^{(0)},c),H^{(0)},E^{(0)})$ be a BBOE of dimension $n$ where $J^{(0)}\!=\!J_n^{(0)}$ and $E^{(0)}\!=\!\{V(x_1),\ldots,V(x_n)\}$. Define the following function $$\begin{array}{cccc} g^{(0)}: &  \ESing(J_n^{(0)},c) & \rightarrow & (\mathbb{Q}\sqcup(\mathbb{Z}\times\mathbb{Q}\times\mathbb{Z}^{\mathbb{N}}) \sqcup\{\infty\}) \\ & \xi & \rightarrow & g^{(0)}(\xi)=t_n^{(0)}(\xi) \end{array}$$ where $t_n^{(0)}$ is defined as in \ref{Einv}, that is, 
$$t_n^{(0)}(\xi)=\left\{\begin{array}{cl} \infty & \text{ if } J_n^{(0)}=1 \\ \frac{\Eord_{\xi}(I_n^{(0)})}{c} & \text{ if } \Eord_{\xi}(I_n^{(0)})>0 \\ \Gamma(\xi)  & \text{ if } \Eord_{\xi}(I_n^{(0)})=0	\end{array}\right.$$

In the case $\Eord_{\xi}(I_n^{(0)})>0$, by induction on the dimension, from the BBOE $B^{(0)}$ we can construct a BBOE in dimension $n-1$, 
 $(W_{n-1}^{(0)},(J_{n-1}^{(0)},c_n^{(0)}),H_{n-1}^{(0)},E_{n-1}^{(0)})$, where  
  \begin{itemize}
	\item $W_{n-1}^{(0)}$ is a hypersurface of $E$-maximal contact for $P_n^{(0)}$. This hypersurface is given by lemma \ref{hiper}. The ideal $P_n^{(0)}$ is the companion ideal of $J^{(0)}=J_n^{(0)}$. 
	\item $J_{n-1}^{(0)}$ is the junior ideal of $P_n^{(0)}$ in $W_{n-1}^{(0)}$.
	\item $c_n^{(0)}=max\ \Eord(P_n^{(0)})$.
	\item $H_{n-1}^{(0)}=H^{(0)}\cap W_{n-1}^{(0)}$.
	\item $E_{n-1}^{(0)}=E^{(0)}\cap W_{n-1}^{(0)}$.
\end{itemize}
	
Define $$\begin{array}{cccl} h^{(0)}: & \EMaxB g^{(0)} & \rightarrow & \mathcal{I}_{n-1} \\ & \xi & \rightarrow & h^{(0)}(\xi)=(t_{n-1}^{(0)}(\xi),\ldots,t_1^{(0)}(\xi)) \end{array}$$ where each function $t_i^{(0)}$, for $i=n-1,\ldots,1$, is defined by induction hypothesis for the BBOE of dimension $i$, $(W_i^{(0)},(J_i^{(0)},c_{i+1}^{(0)}),H_i^{(0)},E_i^{(0)})$ according to \ref{Einv}. 

Extend $h^{(0)}$ to all the points of $\ESing(J^{(0)},c)$: 
\begin{itemize}
	\item[] For each point $\xi=(\xi_1,\ldots,\xi_n)\in W=Spec(K[x,y]_y)$ define the open subset  $U_{\xi}=Spec(K[x,y]_{y,\{x_j/\ \xi_j\neq 0\}})\subset W$. 

Let $a\in \ESing(J^{(0)},c)$ be a point such that $\Eord_{a}(I_n^{(0)})>0$ and $a\notin \EMaxB g^{(0)}$. Fix $a\in U_a$. Assume $x_j(a)=0$ for all $j$, then $a\in \EMaxB g^{(0)}\Big|_{U_a}$. 
\end{itemize}
	
Thus
$$\begin{array}{cccl} t^{(0)}: &  \ESing(J^{(0)},c) & \rightarrow & \mathcal{I}_n \\ & \xi & \rightarrow & t^{(0)}(\xi)=(g^{(0)}(\xi),h^{(0)}(\xi)) \end{array}$$
By construction, $$\EMaxB t^{(0)}=\{\xi\in \ESing(J^{(0)},c)|\ g^{(0)}(\xi)=max\ g^{(0)},\  h^{(0)}(\xi)=max\ h^{(0)}\}$$ $$\hspace*{0.6cm}=\{\xi\in \EMaxB g^{(0)}|\ h^{(0)}(\xi)=max\ h^{(0)}\}=\EMaxB h^{(0)}$$
and it is a permissible center, since by induction hypothesis the result holds for BBOE of dimension $n-1$.

Now assume that by induction on the number of transformations we have defined a sequence of transformations at permissible centers 
\begin{equation} \label{sucesion}
(W^{(0)},(J^{(0)},c),H^{(0)},E^{(0)})\stackrel{\pi_1}{\longleftarrow}\ldots \stackrel{\pi_r}{\longleftarrow}(W^{(r)},(J^{(r)},c),H^{(r)},E^{(r)})
\end{equation}
and a sequence of functions $t^{(1)}=(g^{(1)},h^{(1)}),\ldots,t^{(r-1)}=(g^{(r-1)},h^{(r-1)})$ satisfying the conditions of $B.2$.

To conclude, it is enough to show that if $\ESing(J^{(r)},c)\neq\emptyset$ then the sequence (\ref{sucesion}) extends to a $E$-resolution of $(W^{(0)},(J^{(0)},c),H^{(0)},E^{(0)})$.  

Define the function $$\begin{array}{cccl} g^{(r)}: &  \ESing(J^{(r)},c) & \rightarrow & (\mathbb{Q}\sqcup(\mathbb{Z}\times\mathbb{Q}\times\mathbb{Z}^{\mathbb{N}}) \sqcup\{\infty\}) \\ & \xi & \rightarrow & g^{(r)}(\xi)=\left\{\begin{array}{cl} \Gamma(\xi) & \text{ if } \Eord_{\xi}(I^{(r)})=0 \\ t_n^{(r)}(\xi) & \text{ if } \Eord_{\xi}(I^{(r)})>0  \end{array}\right.\end{array}$$
and the function $$\begin{array}{cccl} h^{(r)}: & \EMaxB g^{(r)} & \rightarrow & \mathcal{I}_{n-1} \\ & \xi & \rightarrow & h^{(r)}(\xi)=(t_{n-1}^{(r)}(\xi),\ldots,t_1^{(r)}(\xi)) \end{array}$$ where each $t_i^{(r)}$ is defined (according to \ref{Einv}) by induction hypothesis for the BBOE of dimension $i$, $(W_i^{(r)},(J_i^{(r)},c_{i+1}^{(r)}),H_i^{(r)},E_i^{(r)})$, with $i=n-1,\ldots,1$. 

As above, extend the function $h^{(r)}$ to all points of $\ESing(J^{(r)},c)$ and set $t^{(r)}=(g^{(r)},h^{(r)})$.

Analogously, $$\EMaxB t^{(r)}=\EMaxB g^{(r)}\cap \EMaxB h^{(r)}=\EMaxB h^{(r)}$$ and by induction on the dimension  $\EMaxB h^{(r)}$ is a permissible center.

Consider the sequence of transformations in dimension $n$, 
\begin{equation} \label{seqd1}
\ldots\stackrel{\pi_{r_0}}{\longleftarrow} (W_{n}^{(r_0)},(J_{n}^{(r_0)},c),H_{n}^{(r_0)},E_{n}^{(r_0)})\stackrel{\pi_{r_0+1}}{\longleftarrow}\ldots \stackrel{\pi_r}{\longleftarrow}(W_{n}^{(r)},(J_{n}^{(r)},c),H_{n}^{(r)},E_{n}^{(r)})
\end{equation}
satisfying $$max\ t_n^{(r_0-1)}>max\ t_n^{(r_0)}=max\ t_n^{(r_0+1)}=\ldots=max\ t_n^{(r)}.$$

And the sequence of transformations in dimension $n-1$ induced by (\ref{seqd1}), 
\begin{equation} \label{seqdim}
(W_{n-1}^{(r_0)},(J_{n-1}^{(r_0)},c_n^{(r_0)}),H_{n-1}^{(r_0)},E_{n-1}^{(r_0)})\!\stackrel{\pi_{r_0+1}}{\longleftarrow}\!\ldots\!\stackrel{\pi_r}{\longleftarrow}\! (W_{n-1}^{(r)},(J_{n-1}^{(r)},c_n^{(r)}),H_{n-1}^{(r)},E_{n-1}^{(r)}).
\end{equation} 
Hence $c_n^{(r_0)}=c_n^{(r_0+1)}=\ldots=c_n^{(r)}$ where $c_n^{(i)}=max\ \Eord(P_n^{(i)})=max\ \Eord(I_n^{(i)})$ with $i=r_0,\ldots,r$. Then the sequence (\ref{seqdim}) in dimension $n-1$ is of the form 
$$(W_{n-1}^{(r_0)},(J_{n-1}^{(r_0)},c_n^{(r_0)}),H_{n-1}^{(r_0)},E_{n-1}^{(r_0)})\!\stackrel{\pi_{r_0+1}}{\longleftarrow}\!\ldots\!\stackrel{\pi_r}{\longleftarrow}\! (W_{n-1}^{(r)},(J_{n-1}^{(r)},c_n^{(r_0)}),H_{n-1}^{(r)},E_{n-1}^{(r)}).$$ 

By induction hypothesis we can extend the sequence (\ref{seqdim}) in dimension $n-1$ to a $E$-resolution of the corresponding BBOE in dimension $n-1$. Suppose this happens after $r$ transformations, that is $\ESing(J_{n-1}^{(r)},c_n^{(r)})=\emptyset$. 

By proposition \ref{saltodim}, after the next transformation, the first component of the $E$-resolution function drops 
$max\ g^{(r)}>max\ g^{(r+1)}$. Note that $max\ g^{(i)}$ can not decrease indefinitely many times because it takes values in $\frac{1}{c}\mathbb{N}$. Then there exists an index $N$ such that the previous sequence (\ref{seqd1}) in dimension $n$ is a $E$-resolution, that is, $\ESing(J_{n}^{(N)},c)=\emptyset$.  
\medskip

Since the $E$-order is an equivariant function (\ref{equitoro}), each $t^{(i)}$ is an equivariant function. Hence the $E$-resolution achieved in this way is invariant by the torus action.
\end{itemize} \vspace*{-0.5cm}
\end{proof}

\begin{remark} Note that this algorithm \ref{resobb} of $E$-resolution of a BBOE $(W,(J,c),H,E)$ is independent of the choice of coordinates and of the choice of the generator system of the ideal $J$.   
\end{remark}

\begin{proposition} {\bf Properties of the algorithm given by Theorem \ref{resobb}.} \label{propis} 
\medskip

Fix a BBOE $(W,(J,c),H,E)$ and a $E$-resolution of this BBOE given by theorem \ref{resobb}. This means  $\ESing(J^{(r)},c)=\emptyset$ for some $r\in\mathbb{N}$, $r>0$.
\begin{enumerate}
	\item If $\xi\in \ESing(J^{(k)},c)$ for $0\leq k\leq r-1$, and $\xi\notin Z^{(k)}$ then  $t^{(k)}(\xi)=t^{(k+1)}(\xi')$ where $\pi_{k+1}(\xi')=\xi$. 
	
That is, it is possible to identify the points in the $E$-singular loci $$\ESing(J^{(0)},c),\ldots,\ESing(J^{(k)},c)$$ and outside the centers $Z^{(0)},\ldots,Z^{(k)},$ with their corresponding transforms in the $E$-singular locus $\ESing(J^{(k+1)},c)$. 
	\item The $E$-resolution is achieved by means of transformations along centers $\EMaxB t^{(k)}$ for $0\leq k\leq r-1$. The $E$-resolution function $t$ drops after each one of these transformations $$max\ t^{(0)}>max\ t^{(1)}>\ldots >max\ t^{(r-1)}.$$
	\item For all $0\leq k\leq r-1$, the closed set $\EMaxB t^{(k)}$ is equidimensional and regular and its dimension is determined by the value $max\ t^{(k)}$.
\end{enumerate}
\end{proposition}

\begin{proof} These properties come from the previous results.
\end{proof}

\begin{remark} Running this algorithm \ref{resobb} of $E$-resolution of a BBOE we only modify the singular points included in the $E$-singular locus. 
\end{remark}

\bibliographystyle{plain}

\providecommand{\bysame}{\leavevmode\hbox to3em{\hrulefill}\thinspace}

\ \\
Roc\'{\i}o Blanco \\
Universidad de Castilla-La Mancha. Departamento de Matem\'aticas. \\ E.U. de Magisterio.
  Edificio Fray Luis de Le\'on.\\ Avda. de los Alfares 42, 16071 Cuenca, Spain.\\
mariarocio.blanco@uclm.es
\noindent 

\end{document}